\documentclass[11pt]{article}

\usepackage{amssymb,graphics,amsmath,amsthm,amsopn,amstext,amsfonts,color,fullpage,fancybox,hyperref,bm}
\usepackage{subfig,placeins} 
\usepackage{graphicx,color,float,epstopdf}
\usepackage{fix-cm}
\usepackage{arydshln}
\usepackage{mathtools}
\usepackage{graphicx}
\usepackage{psfrag}
\usepackage{epsfig}
\usepackage{graphics,bm,paralist,ifthen,color, algorithm}
\usepackage{array}
\usepackage{algpseudocode}
\usepackage{calc}
\usepackage{longtable}
\usepackage{mathrsfs}
\usepackage{float}
\usepackage{lipsum}
\usepackage{epstopdf}

\newtheorem{theorem}{Theorem}

\newtheorem{lemma}[theorem]{Lemma}

\newtheorem{definition}[theorem]{Definition}

\newtheorem{remark}[theorem]{Remark}
\newtheorem{assumption}[theorem]{Assumption}

\makeatletter
\def\BState{\State\hskip-\ALG@thistlm}
\makeatother

\algnewcommand{\IfThenElse}[3]{
	\State \algorithmicif\ #1\ \algorithmicthen\ #2\ \algorithmicelse\ #3}

\newcommand{\gc}{{\|\mathbf{g}_k +\mathbf{A}^T\bm{\lambda}_k^C \|_2}}

\newcommand{\st}{\rm s.t.}

\newcommand{\cX}{{\cal X}}

\newcommand{\bzero}{\mathbf{0}}
\newcommand{\cpsi}{{\cal \psi}}

\newcommand{\blambda}{\boldsymbol{\lambda}}

\newcommand{\ba}{\mathbf{a}}
\newcommand{\bb}{\mathbf{b}}

\newcommand{\bd}{\mathbf{d}}

\newcommand{\bg}{\mathbf{g}}

\newcommand{\bs}{\mathbf{s}}

\newcommand{\bx}{\mathbf{x}}

\newcommand{\bA}{\mathbf{A}}

\newcommand{\bH}{\mathbf{H}}
\newcommand{\bI}{\mathbf{I}}

\title{A Trust Region Method for Finding Second-Order Stationarity in Linearly Constrained Non-Convex Optimization}

\author{Maher  Nouiehed	\and Meisam Razaviyayn\thanks{Department of Industrial and Systems Engineering, University of Southern California.}}

\begin{document}
	
	\maketitle
	
	\begin{abstract}
		\noindent Motivated by TRACE algorithm~\cite{curtis2017trust}, we propose  a trust region algorithm for finding second order stationary points of a linearly constrained non-convex optimization problem. We show the convergence of the proposed algorithm to ($\epsilon_g, \epsilon_H$)-second order stationary points in   $\widetilde{\mathcal{O}}\left(\max\left\{\epsilon_g^{-3/2}, \epsilon_H^{-3}\right\}\right)$ iterations. This iteration complexity is  achieved for general linearly constrained optimization without cubic regularization of the objective function. 
	\end{abstract}
	
	\section{Introduction}
	
	Due to its wide application in machine learning, solving non-convex optimization problems encountered significant attention in recent years \cite{anandkumar2016efficient,cartis2013evaluation, cartis2015evaluation,  cartis2017second, curtis2017inexact,curtis2017trust, lu2012trust}. While this topic has been studied for decades, recent applications and modern analytical and computational tools revived this area of research. In particular, a wide variety of numerical methods for solving non-convex problems have been proposed in recent years \cite{lee2016gradient, hong2018gradient, nesterov2006cubic, cartis2012adaptive, cartis2011adaptive, cartis2011adaptive-2,nouiehed2019solving}. \\
	
	For general non-convex optimization problems, it is well-known that computing a local optimum is NP-Hard \cite{murty1987some}. Given this hardness result, recent focus has been shifted toward computing (approximate) first and second-order stationary points of the objective function. The latter set of points provides stronger guarantees compared to the former as it constitutes a smaller subset of points that includes local and global optima.  Therefore, when applied to problems with ``\textit{nice}'' geometrical properties, the set of second order stationary points could even coincide with the set of global optima -- see \cite{bandeira2016low, barazandeh2018behavior, boumal2016nonconvex, nouiehed2018learning, sun2016geometric, ge2015escaping,sun2017complete-2,  sun2017complete} for examples of such objective functions.\\

	Convergence to second-order stationarity in  smooth unconstrained setting has been thoroughly investigated in the optimization literature \cite{ge2016matrix, curtis2017exploiting, cartis2011adaptive, cartis2011adaptive-2, nesterov2006cubic, curtis2017inexact, curtis2017trust, ge2015escaping}. As a second-order algorithm, \cite{nesterov2006cubic} proposed a cubic regularization method that converges to approximate second-order stationarity in finite number of steps. More recently, \cite{cartis2011adaptive, cartis2011adaptive-2} proposed the Adaptive Regularization Cubic algorithm (ARC) that computes an approximate solution for a local cubic model at each iteration. They established convergence to  first and second order stationary points with optimal complexity rates. Motivated by these rates, \cite{curtis2017trust} proposed an adaptive trust region method, entitled TRACE, and established iteration complexity bounds for finding $\epsilon$-first-order stationarity with worst-case  iteration complexity ${\cal O}(\epsilon^{-3/2})$; and for finding   $(\epsilon_g,\epsilon_H)$-second-order stationarity with worst-case complexity ${\cal O}(\max\{\epsilon_g^{-3/2}, \epsilon_H^{-3}\})$. This method alters the acceptance criteria adopted by traditional trust region methods, and implements a new mechanism for updating the trust region radius. A more recent second-order algorithm that uses a dynamic choice of direction and step-size was proposed in \cite{curtis2017exploiting}. This method computes first and second order descent directions and chooses the direction that predicts a more significant reduction in the objective value. All of the above methods satisfy the set of generic conditions of a general framework proposed in \cite{curtis2017inexact}.\\
	
	Recent results show that for smooth unconstrained optimization problems, even first order methods can converge to second-order stationarity, almost surely. For instance, \cite{ge2015escaping} shows that noisy stochastic gradient descent escapes \textit{strict saddle points} with probability one. Therefore, when applied to problems satisfying the \textit{strict saddle property} this method converges to a local minimum. A similar result was shown for the vanilla gradient descent algorithm in \cite{lee2016gradient}. 
	A negative result provided by \cite{du2017gradient} shows that vanilla gradient descent can take exponential number of steps to converge to second-order stationarity. This computational inefficiency can be overcome by a smart perturbed form of gradient descent proposed in \cite{jin2017escape}. \\
	
	Most of the above results can be extended to the smooth constrained optimization in the presence of simple manifold constraints. In this case, \cite{lee2017first} shows that manifold gradient descent converges to second-order stationarity, almost surely. More recently, \cite{hong2018gradient} established similar results for gradient primal-dual algorithms applied on linearly constrained optimization problems. When the constraints are non-manifold type, projected gradient descent is a natural replacement of gradient descent. As a negative result, \cite{nouiehed2018convergence} constructs an example, with a single linear constraint, showing that there is a positive probability that projected gradient descent with random initialization can converge to a strict saddle point. This raises the question of
	\textit{whether there exist a first order method that can converge to second-order stationarity in the presence of inequality constraints.} To our knowledge, no affirmative answer has been given to this question to date.\\
	
	The answer to the question above is obvious when replacing first-order methods with second-order methods. In fact, convergence to second-order stationarity in the presence of convex constraints has been established by adapting many of the aforementioned second-order algorithms \cite{cartis2012adaptive, cartis2015evaluation, cartis2014complexity}. The work in \cite{cartis2012adaptive} adapts the ARC algorithm and showed convergence to $\epsilon_g$-first-order stationarity in at most ${\cal O}(\epsilon_g^{-3/2})$ iterations. \cite{birgin2018regularization} uses active set method and cubic regularization to achieve this rate for special types of constraints. The work\cite{bian2015complexity} uses interior point method to achieve a second order stationarity in $O\big(\max\{\epsilon_g^{-3/2}, \epsilon_H^{-3}\}\big)$ iterations for box constraints. For general constraints, \cite{cartis2017second} proposed a \textit{conceptual} trust region algorithm that can compute an $\epsilon$-${q^{th}}$ stationary point in at most ${\cal O}(\epsilon^{-q-1})$ iterations. 
	More recently, \cite{mokhtari2018escaping} proposed a general framework for computing $(\epsilon_g, \epsilon_H)$-second-order stationary points for convex-constrained optimization problem with worst-case complexity ${\cal O}\big(\max\{\epsilon_g^{-2}, \epsilon_H^{-3}\}\big)$. In particular, this framework allows for using Frank-Wolfe or projected gradient descent to converge to an approximate first-order method, and then computes a second-order descent direction if it exists. The iteration complexity bounds computed for these methods hide the per-iteration complexity of solving the quadratic or cubic sub-problems. As shown in \cite{nouiehed2018convergence}, for linearly constrained non-convex problems, even checking whether a given point is an \textit{approximate} second-order stationary point is NP-Hard. Despite this hardness result, \cite{nouiehed2018convergence} proposed a second-order Frank-Wolfe algorithm that adapts the dynamic method introduced in \cite{curtis2017exploiting}, and identified instances for which solving the constrained quadratic sub-problem can be done efficiently. The algorithm converges to approximate first and second-order stationarity with a worst-case complexity similar to \cite{curtis2017exploiting}. 
	However, second-order information as utilized in the adapted ARC algorithm yields better iteration complexity rates. Motivated by this result, in this paper, we propose a trust region algorithm, entitled LC-TRACE, that adapts TRACE to linearly-constrained non-convex problems. We establish the  convergence of our algorithm to $(\epsilon_g,\epsilon_H)$-second order stationarity in at most $\widetilde{\cal O}(\epsilon_g^{-3/2}, \epsilon_H^{-3})$ iterations. \\
	
	The remainder of this paper is organized as follow. In section~\ref{sec:FirstSecondOS}, we first review and define the concepts of first and second order stationarity. Then, we review some of our previous results in section~\ref{sec:ReviewPrevious}. Finally, in section~\ref{sec:LC-TRACE}, we propose and analyze LC-TRACE algorithm.
	

	\section{First and Second Order Stationarity Definitions}
	\label{sec:FirstSecondOS}
	To understand the definition of first and second order stationarity, let us first start by considering the unconstrained optimization problem
	\begin{equation}\label{eq:General-Optimization-Prob}
	\underset{\bx \in \mathbb{R}^{n}}{\min} \, \, f(\bx),
	\end{equation}
	where $f: \mathbb{R}^n \mapsto \mathbb{R}$ is a twice continuously differentiable function. We say a point $\bar\bx$ is a first order stationary point (FOSP) of~\eqref{eq:General-Optimization-Prob} if $\nabla f(\bar\bx) = \bzero$. Similarly, a point $\bar\bx$ is said to be a second-order stationary point (SOSP) of~\eqref{eq:General-Optimization-Prob} if $\nabla f(\bar\bx) = \bzero$ and $\nabla^2 f(\bar\bx) \succeq 0$. In practice, most of the algorithms used for finding stationary points are iterative. Therefore, we define the concept of approximate first and second order stationarity. We say a point $\bar\bx$ is an 
	$\epsilon_g$-first-order stationary point if 
	\begin{equation}\label{eq:FOSUnconstrained}
	\|\nabla f(\bar\bx)\|_2 \leq \epsilon_g.
	\end{equation}
	Moreover, we say a point $\bar\bx$ is an 
	$(\epsilon_g, \epsilon_H)$-second-order stationary point if 
	\begin{equation}\label{eq:SoSUnconstrained}
	\|\nabla f(\bar\bx)\|_2 \leq \epsilon_g \mbox{   and   } \nabla^2 f(\bar\bx) \succeq -\epsilon_H\bI.
	\end{equation}
	

	We now extend these definitions to the constrained optimization problem
	\begin{equation}\label{eq:General-Cons-Optimization-Prob}
	\underset{\bx \in \mathcal{P}}{\min} \, \, f(\bx),
	\end{equation}
	where $\mathcal{P}\subseteq \mathbb{R}^n$ is a closed convex set. As defined in \cite{bertsekas1999nonlinear}, we say $\bar{\bx} \in {\cal P}$ is a FOSP of \eqref{eq:General-Cons-Optimization-Prob} if 
	\begin{equation}\label{eq:FOSConstrained}
	\langle \nabla f(\bar{\bx}), \bx - \bar{\bx} \rangle \geq 0 \quad \forall \, \bx \in \mathcal{P}.
	\end{equation}
	Similarly, we say a point $\bar{\bx}$ is a SOSP of the optimization problem \eqref{eq:General-Cons-Optimization-Prob} if $\bar{\bx} \in {\cal P}$ is a first order stationary point and
	\begin{equation}\label{eq:SoSConstrained}
	0\leq \bd^T \nabla^2f(\bar{\bx}) \bd,\quad \forall \, \bd \,\,\st\,\, \langle\bd,\nabla f(\bar{\bx}) \rangle =0\textrm{ and } \bar{\bx} + \bd \in \mathcal{P}. 
	\end{equation}
	
	Notice that when ${\cal P}=\mathbb{R}^n$, the definitions above obviously correspond to the definitions in the unconstrained case.\\
	
	Motivated by \eqref{eq:FOSConstrained} and \eqref{eq:SoSConstrained}, given a feasible point $\bx$, we define the following first and second order stationarity measures
	\begin{equation}\label{eq:X_k}
	\begin{split}
	\cX(\bx) \triangleq -  \;\min_{\bs}\quad &\langle \nabla f(\bx), \bs \rangle   \\
	\st \quad & \bx + \bs \in {\cal P}, \, \|\bs\|\leq 1.
	\end{split}
	\end{equation}
	and 
	\begin{equation}\label{eq:psi_k}
	\begin{split}
	{\cal \psi}(\bx) \triangleq - \; \min_{\bd}\quad &\bd^T \nabla^2f(\bx)\bd \,    \\
	\st \quad & \bx + \bd \in {\cal P},  \, \|\bd\|\leq 1\\ &\langle \nabla f(\bx), \bd \rangle \leq 0.
	\end{split}
	\end{equation}

	Notice that since $\bx$ is feasible,  $\cX(\bx) \geq 0 $ and ${\cal \psi}(\bx) \geq 0$. Moreover, these optimality measures, which are also used in \cite{nouiehed2018convergence}, can be linked to the standard definitions in \cite{bertsekas1999nonlinear} by the following Lemma. 
	\begin{lemma}[\cite{nouiehed2018convergence}] \label{lem:StationarityContinuous}
		The first and second order stationarity measures $\cX(\cdot)$ and $\cpsi(\cdot)$ are continuous in $\bx$. Moreover, if $\bar\bx \in {\cal P}$ then
		\begin{itemize}
			\item $\cX (\bar\bx) = 0$ if and only if $\bar\bx$ is a first order stationary point.
			\item $\cX(\bar{\bx}) = \cpsi(\bar{\bx}) = 0$ if and only if $\bar\bx$ is a second order stationary point.
		\end{itemize}
	\end{lemma}

	\vspace{0.2cm}
	Using this lemma, we define the approximate first and second order stationarity.
	
	\vspace{0.1cm}
	
	\begin{definition}\label{def:EpsFOSSOS}
		\textbf{Approximate Stationary Point:} For problem \eqref{eq:General-Cons-Optimization-Prob}, \begin{itemize}
			\item A point $\bar\bx \in {\cal P}$ is said to be an $\epsilon_g$-first order stationary point if $\cX(\bar\bx) \leq \epsilon_g$. 
			\item A point $\bar\bx \in {\cal P}$ is said to be an $(\epsilon_g, \epsilon_H)$-second order stationary point if $\cX(\bar\bx) \leq \epsilon_g$ and $\cpsi(\bar{\bx}) \leq \epsilon_H$.  
		\end{itemize} 
	\end{definition}
	
	\vspace{0.2cm}
	
	In the unconstrained scenario, these definitions correspond to the standard definitions \eqref{eq:SoSUnconstrained} and \eqref{eq:FOSUnconstrained}.
	
	\vspace{0.2cm}
	
	\begin{remark}
		Notice that our definition of $(\epsilon_g,\epsilon_H)$-second order stationarity is different than the definition in \cite{mokhtari2018escaping}. In particular, there are two major differences:
		\begin{itemize}
			\item[1)] The definition used for approximate first and second order stationarity in \cite{mokhtari2018escaping} does not include the normalization constraints $\|\bs\|\leq 1$ and $\|\bd\|\leq 1$ in \eqref{eq:X_k} and \eqref{eq:psi_k}. 
			\item[2)] The  second order optimality measure in \cite{mokhtari2018escaping} is defined based on using equality constraint $\langle \nabla f(\bx), \bd \rangle = 0$ in~\eqref{eq:psi_k} instead of the inequality constraint $\langle \nabla f(\bx), \bd \rangle \leq 0$. 
		\end{itemize}
		To understand the necessity of using normalization, consider the optimization problem $\min x^2$ and the point $\bar{x} = \epsilon$ with $\epsilon$ being (arbitrary) small. Clearly, $\bar{x}$ is close to optimal, while the optimality measure \eqref{eq:X_k} does not reflect this approximate optimality if we do not include the normalization constraint in \eqref{eq:X_k}. 
		
		To understand the importance of using inequality constraint $\langle \nabla f(\bx), \bd \rangle \leq 0$ instead of  equality constraint in \eqref{eq:psi_k}, consider the scalar optimization problem 
		\begin{align}
		\min_x \;\; &-\frac{1}{2}  x^2 \nonumber\\
		\st \quad &0\leq x\leq 10. \nonumber
		\end{align}
		Let us look at the point $\bar{x}  = \epsilon >0$. Using second order information, one can  say that $\bar{x}$ is not a reasonable point to terminate your algorithm at. This is because the Hessian provides a descent direction with large amount of improvement in the second order approximation of the objective value. This fact is also reflected in the value of $\psi(\bar{x})= 1$. However, if we had used equality constraint $\langle \nabla f(\bx), \bd \rangle =  0$ in the definition of $\psi(\cdot)$ in \eqref{eq:psi_k}, then the value of $\psi(\cdot)$ would have been zero.
	\end{remark}
	
	\vspace{0.2cm}
	
	\begin{remark}
		There are other definitions of second order stationarity in the literature. For example, the works~\cite{bian2015complexity, birgin2018regularization} use a scaled version of the Hessian in different directions to define second order stationarity for box constraints. Recently, \cite{o2019log} carefully revised it to account for the coordinates which are very far from the boundary. Another related definition of second order stationary, which leads to a practical perturbed gradient descent algorithm, is provided in \cite{SongTao2019} for general linearly constrained optimization problems.
	\end{remark}

	\section{Finding second-order stationary points for constrained optimization}
	\label{sec:ReviewPrevious}
	Consider the quadratic co-positivity problem
	\begin{equation}\label{co-positivity-problem}
	\min_{\bx \in \mathbb{R}^n} \quad \frac{1}{2} \bx^T \mathbf{Q}\bx \quad \quad \st \quad  \bx \geq \bzero, \,\,\|\bx\| \leq 1.
	\end{equation}
	Clearly, checking whether $\bar\bx = \bzero$ is a second order stationary point of~\eqref{co-positivity-problem} is equivalent to checking its local optimality, which is an NP-Hard problem \cite{murty1987some}. This observation shows that checking exact second order stationarity is Hard. The following result, which is borrowed from~\cite{nouiehed2018convergence}, shows that even checking \textit{approximate} second order stationarity is NP-hard. 
	
	\begin{theorem} [Theorem~6 in \cite{nouiehed2018convergence}]
		There is no algorithm which can check whether x = 0 is an $(\epsilon_g,\epsilon_H)$-second order stationary point in polynomial time in $(n, 1/{\epsilon_H})$, unless P = NP.
	\end{theorem}
	This hardness result implies that we should not expect  an efficient algorithm for finding second order stationary points of non-convex problems. However, in this problem, \textit{the source of hardness stems from the number of linear inequality constraints}. In fact, when we have a small constant number of linear constraints (say $m$ fixed constraints), \cite{hsia2013trust} proposed a backtracking approach that efficiently solves this quadratic constrained optimization problem. Although the method proposed is exponential in $m$ as it uses an exhaustive search over the set of active constraints, it can still be used when $m$ is small. Motivated by this observation, in the next section we describe our LC-TRACE algorithm and analyze its iteration complexity for finding second order stationary points of linearly constrained non-convex optimization problems. A core assumption in our algorithm is that a certain quadratic objective can be minimized given existing linear constraints (for example when $m$ is small).
	
	\section{A Trust Region Algorithm for Solving Linearly-Constrained Smooth Non-Convex Optimization Problems}\label{sec:LC-TRACE}
	Consider the optimization problem
	\begin{equation}\label{Linearly-Cons-Prob}
	\begin{split}
	\min_{\bx \in \mathbb{R}^{ n}} \, \, & f(\bx), \\
	\st \quad & \bA  \bx \leq \bb,
	\end{split}
	\end{equation}
	where $\bA \in \mathbb{R}^{m\times n}$ and $\bb \in \mathbb{R}^m$.
	In this section, we propose a trust region algorithm, entitled LC-TRACE (Linearly Constrained TRACE), that adapts TRACE \cite{curtis2017trust} to the above linearly-constrained non-convex problem. We establish its convergence to $\epsilon_g$-first order stationarity with iteration complexity order $\widetilde{\cal O}(\epsilon_g^{-3/2})$. This method is then used to develop an algorithm to converge to~$(\epsilon_g, \epsilon_H)$-second-order stationarity with the iteration complexity $\widetilde{\cal O}\big(\max\{\epsilon_g^{-3/2}, \epsilon_H^{-3}\}\big)$.\\

	LC-TRACE is different from the traditional trust region method proposed in \cite{cartis2017second} for constrained optimization. More specifically, LC-TRACE utilizes the mechanisms used in TRACE \cite{curtis2017trust} to provide a faster convergence rate compared to \cite{cartis2017second}. The improved convergence rate matches the rates achieved by adapted ARC \cite{cartis2012adaptive} and TRACE \cite{curtis2017trust}, up to logarithmic factors. Since applying TRACE directly to constrained optimization fails (as will be discussed later), we introduced modifications to adapt this method to linearly constrained problems. Our modifications are not the result of a ``simple extension" of unconstrained  to constrained scenario.  Before explaining LC-TRACE, let us first provide an overview of the classical trust region and TRACE algorithms. 
	
	\subsection{Background on Traditional Trust Region Algorithm and TRACE} 
	In traditional trust region methods, the trial step $\bs_k$ at iteration $k$ is computed by solving the standard trust region  sub-problem
	\begin{equation}\label{sub-problem}
	\displaystyle{ \operatornamewithlimits{\mbox{min}}_{\bs \in \mathbb{R}^n }} \, \, q_k(\bs), \quad \mbox{s.t. } \|\bs\|_2 \leq \delta_k,
	\end{equation}
	where $q_k(\bs): \mathbb{R}^n \mapsto \mathbb{R}$ is the second-order Taylor  approximation of $f$ around $\bx_k$, i.e., 
	\[
	q_k(\bs) \triangleq f_k + \bg_k^T\bs + \dfrac{1}{2}\bs^T\bH_k \bs.
	\]
	Here $f_k = f(\bx_k)$, $\bg_k = \nabla f(\bx_k)$, and $\bH_k = \nabla^2 f(\bx_k)$. Based on the resulting trial step, an acceptance criteria is used to either \textit{accept} or \textit{reject} the step. In particular, if the ratio of actual-to-predicted reduction 
	\[ \dfrac{f_k - f(\bx_k + \bs_k)}{f_k - q_k(\bs_k)}\]
	is greater than a prescribed constant, the step is accepted, otherwise it is rejected. The iterate $\bx^{k+1}$ and trust region radius are updated accordingly. Traditional trust region methods use a geometric update rule for the trust region radius $\delta_k$, i.e., $\delta_{k+1}$ is some constant factor of $\delta_k$. TRACE algorithm, on the other hand, modifies the acceptance criteria and this linear update rule for $\delta_k$ to match the rate achieved by the ARC algorithm \cite{cartis2011adaptive, cartis2011adaptive-2}. In particular, the authors in \cite{curtis2017trust} observed that ARC computes a positive sequence of cubic regularization coefficients $\sigma_k \in [\underline{\sigma}, \,  \overline{\sigma}]$ that satisfy 
	\begin{equation}\label{eq:alg-conditions}
	f_k - f_{k+1} \geq c_1\sigma_{k}\|\bs_k\|_2^3 \quad \mbox{and} \quad \|\bs_k\|_2 \geq \Big(\dfrac{c_2}{\overline{\sigma} + c_3}\Big)^{1/2}\|\bg_{k+1}\|_2^{1/2},
	\end{equation} 
	for some given positive constants $c_1, c_2, c_3$. TRACE designed a modified acceptance criteria and a new mechanism for updating the trust region radius to satisfy the conditions provided in \eqref{eq:alg-conditions}. Some of these ideas are discussed next.\\

	\textbf{Sufficient Decrease Acceptance Criteria.}
	TRACE defines the ratio
	\begin{equation}\label{eq:rho-k}
	\rho_k \triangleq \dfrac{f_k - f(\bx_k + \bs_k)}{\|\bs_k\|_2^3},
	\end{equation}
	as a measure to decide whether to \textit{accept} or \textit{reject} a trial step. For some prescribed $\rho \in (0,1)$, a trial step $\bs_k$ can only be accepted if $\rho_k \geq \rho$. By noticing that a small $\|\bs_k\|_2$ may satisfy only the first condition in \eqref{eq:alg-conditions}, the developers of TRACE realize that an acceptance criteria that only involves \eqref{eq:rho-k} is not sufficient. To avoid such cases, TRACE defines a sequence $\{\sigma_k\}$ to estimate an upper bound for the ratio $\lambda_k/\|\bs_k\|_2$ used for acceptance. Here $\{\lambda_k\}$ is the sequence of dual variables corresponding to the constraint $\|\bs\|_2 \leq \delta_k$ in sub-problem~\eqref{sub-problem}. In short, TRACE accepts a trial pair $(\bs_k, \lambda_k)$ if it satisfies the following conditions:
	\begin{equation}\label{acceptance-criteria}
	\rho_k \geq \rho \quad \mbox{and} \quad  \lambda_k /\|\bs_k\|_2 \leq \sigma_k.
	\end{equation}

	\textbf{Trust Region Radius Update Procedure.} In contrast to the linear update rule utilized in traditional trust region algorithms, TRACE uses a CONTRACT subroutine that allows for sub-linear updates. In particular, this subroutine compares the radius obtained by the linear update scheme to the norm of the trial step computed using 
	\begin{equation}\label{reg-sub-problem}
	\displaystyle{ \operatornamewithlimits{\mbox{min}}_{\bs \in \mathbb{R}^n }} \, \, f_k + \bg_k^T\bs + \dfrac{1}{2}\bs^T(\bH_k + \lambda \bI)\bs,
	\end{equation}
	for a carefully chosen $\lambda$. If the norm of this trial step falls within a desired range, then it is chosen to be the new trust region radius. This subroutine is called at iteration $k$ if $\rho_k < \rho$.\\

	TRACE is designed to solve unconstrained smooth optimization problems. A direct implementation of this algorithm fails in the constrained setting. In the next section, we describe \textit{two fundamental difficulties} introduced in the presence of constraints and discuss the necessary modifications. 

	\subsection{Difference Between LC-TRACE and TRACE}\label{Section:Diff-Bw-Alg}
	In the constrained setting, we define the trust region sub-problem and its regularized Lagrangian form as
	\begin{equation}\label{sub-problem-cons}
	Q_k \triangleq \displaystyle{ \operatornamewithlimits{\mbox{min}}_{\bs \in \mathbb{R}^n }} \, \, q_k(\bs), \quad \mbox{s.t. } \begin{cases} 
	\bA\bs \leq \bb - \bA\bx_k\\
	\|\bs\|_2 \leq \delta_k
	\end{cases}
	\end{equation}
	and 
	\begin{equation}\label{reg-sub-problem-cons}
	Q_k(\lambda) \triangleq \displaystyle{ \operatornamewithlimits{\mbox{min}}_{\bs \in \mathbb{R}^n }} \, \, f_k + \bg_k^T\bs + \dfrac{1}{2}\bs^T(\bH_k + \lambda \bI)\bs, \quad \mbox{s.t. } 
	\bA\bs \leq \bb - \bA\bx_k.
	\end{equation}
	A major difficulty introduced by the constraints is related to the optimality conditions of the sub-problem. In the unconstrained case, it is known that $\bH_k + \lambda_k\bI \succeq 0$ at every iteration \cite[Corollary~7.2.2]{conn2000trust} for optimal Lagrange multiplier~$\lambda_k$. Along with the fact that $\lambda > \lambda_k$ in the CONTRACT subroutine of TRACE, we conclude that $Q_k({\lambda})$ is a strongly convex quadratic optimization problem which has a unique global minimizer. Let $\bs^{*}(\lambda)$ be the solution of $Q_k(\lambda)$. It follows that the function $\bs^{*}(\lambda)$ is continuous in $\lambda$ in the unconstrained scenario. However,
	in the linearly constrained scenario, the regularized sub-problem \eqref{reg-sub-problem-cons} might have multiple optimal solutions. Moreover, $\bs^*(\lambda)$ and the ratio $\lambda/\|\bs^{*}(\lambda)\|_2$, which are core quantities in TRACE, might not even be continuous. To clarify this difficulty, consider the following simple example
	\begin{equation}\label{example}
	Q(\lambda) = \underset{s_1 \leq 5, \, s_2 \geq 0 }{\min} \, \, s_1^2 - s_2^2 + \lambda(s_1^2 + s_2^2) \quad \mbox{s.t. } s_2 - 3s_1 \leq -12.
	\end{equation}
	It is not hard to see that the optimal solution of~\eqref{example} is given by
	\[ s^{*}(\lambda) = \begin{cases} \begin{array}{cc}
	(5,3) \quad  & \mbox{if } \lambda<0   \\
	(5,3); \, (4,0) \quad  & \mbox{if } \lambda=0 \\
	(4,0) \quad  & \mbox{Otherwise.}  
	\end{array}\end{cases}.\]
	Thus, a small increases in $\lambda$ may lead to a huge change in the ratio $\lambda / \|s^{*}(\lambda)\|_2$. Therefore, the luxury of having an arbitrarily choice for the bounds $\underline{\sigma}$ and $\bar{\sigma}$ of the ratio $\lambda/\|\bs\|_2$ is not present in the constrained case. In LC-TRACEC, we resolved this issue by defining 
	\begin{equation}\label{sigma-min-max}
	\underline{\sigma} = \dfrac{\epsilon}{C_{min} + \max\{\lambda_{max}, \lambda_0\}} \quad \mbox{and} \quad \bar{\sigma} = 2\Delta,
	\end{equation}
	and altering the update rule of $\lambda$ in the CONTRACT sub-routine. Here $\epsilon>0$ is the threshold used for the termination of the algorithm,  $C_{min}$ is defined in Lemma \ref{LXksk}, $\lambda_{max}$ is defined in Lemma \ref{LBG}, and $\Delta$ is defined in Lemma~\ref{L3.11}. \\
	
	Another major difficulty in the  constrained scenario is related to the standard trust region theory on the relationship between sub-problem solutions and their corresponding dual variables. In the unconstrained case,  $\lambda_1 > \lambda_2$, implies $\|\bs^{*}(\lambda_1)\|_2 < \|\bs^{*}(\lambda_1)\|_2$, see \cite[Chapter~7]{conn2000trust}. This relationship was used in \cite{curtis2017trust}, to show that CONTRACT subroutine reduces the radius of the trust region. However, it can be seen from example~\eqref{example}, that this relation may not hold in the constrained case. To account for this issue, we modified the CONTRACT sub-routine to guarantee a reduction in the trust region radius (See Lemma \ref{L3.4}). In summary, the differences between LC-TRACE and TRACE are mainly in the CONTRACT sub-routine. Next, we describe the steps of the algorithm.

	\subsection{Description of LC-TRACE}
	Our proposed algorithm LC-TRACE has two main building blocks: \textit{First-Order-LC-TRACE} and \textit{Second-Order-LC-TRACE}. We first present First-Order-LC-TRACE which can converge to $\epsilon_g$-first order stationarity in $\mathcal{\tilde{O}} (\epsilon_g^{-3/2})$. Then, we use this algorithm in Second-Order-LC-TRACE to find an $(\epsilon_g,\epsilon_H)$-Second Order stationarity in $\mathcal{\tilde{O}}(\max\{\epsilon_g^{-3/2},\epsilon_H^{-3}\})$ iterations.
	
	The First-Order-LC-TRACE algorithm is outlined in Algorithm~\ref{alg-LC}. At each iteration $\bx_k$,  this iterative algorithm computes the values $\bs_k$, $\blambda_k$, and $\rho_k$  by solving the optimization problem~\eqref{sub-problem-cons} and using the equation~\eqref{eq:rho-k}.  Depending on the obtained values, it decides to either \textit{accept} the trial point $\bs_k$, or reject it. When rejecting the trial point, it either goes to \textit{contraction} or \textit{expansion} procedures. Thus, the main decisions include: Acceptance, Contraction, or Expansion. We distinguish the iterations by partitioning the set of iteration numbers into what we refer to as the sets of accepted $({\cal A})$, contraction $({\cal C})$, and expansion $({\cal E})$ steps:
	\begin{flalign*}
	&{\cal A} \triangleq \{k \in \mathbb{N} : \rho_k \geq  \rho \mbox{ and either } \lambda_k \leq {\sigma}_k\|\bs_k\|_2 \mbox{ or }  \|\bs_k\|_2 = \Delta_k \}, \\
	&{\cal C} \triangleq \{k \in \mathbb{N}: \rho_k  < \rho\}, \mbox{ and}\\
	&{\cal E} \triangleq \{k \in \mathbb{N} : k \notin {\cal A} \, \cup \, {\cal C}\}.
	\end{flalign*}
	
	Hence, step $k$ is accepted if the computed pair $(\bs_k,\lambda_k)$ satisfies the sufficient decrease criteria $\rho_k \geq \rho$, and either the norm of $\bs_k$ is large enough $(\|\bs_k\|_2=\Delta_k)$ or the ratio $\lambda_k/\|\bs_k\|_2$ is smaller than an upper-bound $\sigma_k$. We also partition the set of accepted steps into two disjoint subsets
	\[{\cal A}_{\Delta} \triangleq \{k \in {\cal A} : \|\bs_k\|_2 = {\Delta}_k \} \mbox{ and } {\cal A}_{\sigma} \triangleq \{k \in {\cal A} : k \notin {\cal A}_{\Delta}\}.
	\]
	The sequence ${\Delta_k}$ is used in the algorithm as an upper bound on the norm of $\|\bs_k\|_2$. From steps~\ref{Delta-acc}, \ref{Delta-cont},  and \ref{Delta-exp}, we notice that this sequence is non-decreasing. We now describe the update mechanism used in a contraction step of First-Order-LC-TRACE which is the main difference between TRACE and our proposed algorithm.\\ 
	
	When CONTRACT subroutine is called, two different cases may occur in Algorithm~\ref{alg-contract}. The first case is reached whenever conditions in Step~\ref{Cond-lower-bound-no} in the CONTRACT subroutine tests true. In that case, we carefully choose choose $\lambda > \lambda_k$
	to ensure that the pair $(\bs, \lambda)$ with $\bs$ being the solution of $Q_k(\lambda)$ satisfies
	\begin{equation*}
	\underline{\sigma}\leq \lambda/\|\bs\|_2 \leq \overline{\sigma}, 
	\end{equation*}
	where $\underline{\sigma}$ and $\bar{\sigma}$ are prescribed positive constants defined in \eqref{sigma-min-max}. The second case is reached whenever the conditions in Step~\ref{Cond-lower-bound-no} tests false. In that case, we choose $\lambda \in (\lambda_k, C\lambda_k]$ with $C>1$ is a constant scalar, to ensure that the pair $(\bs, \lambda)$ with $\bs$ being the solution of $Q_k(\lambda)$ satisfies the following
	\begin{equation*}
	\dfrac{\lambda}{\|\bs\|_2} < \max\Bigg\{\bar{\sigma}, \Big(\dfrac{\gamma_{\lambda}}{\gamma_C}\Big)\dfrac{H_{Lip} + 2\rho}{2\kappa}\Bigg\},
	\end{equation*}
	where $\kappa \in (0,1]$ is a constant scalars, and $H_{Lip}$ is defined in assumption \ref{Assumption1-LC}. In what follows, we first present  our  results about the convergence of First-Order-LC-TRACE algorithm and its iteration complexity.
	
	\begin{algorithm}
		\caption{First-Order-LC-TRACE}\label{alg-LC}
		\textbf{Require:} an acceptance constant $\rho \in (0,1)$.\\
		\textbf{Require:} update constants $\{\gamma_C, \gamma_E, \gamma_{\lambda} \}$ with $\gamma_C \in (0,1)$ and $\gamma_{\lambda}, \gamma_E >1$.\\
		\textbf{Require:} ratio bound constants $\underline{\sigma}$ and $\overline{\sigma}$ defined in \eqref{sigma-min-max}.
		
		\hrulefill
		
		\begin{algorithmic}[1]
			
			\Procedure{First-Order-LC-TRACE}{}
			
			\State Choose a feasible point $\bx_0$,  a pair $\{\delta_0 , \Delta_0 \}$ with $0<\delta_0 \leq \Delta_0$, and $\sigma_0$ with $\sigma_0 \geq \underline{\sigma}$. \label{initial}
			
			\State Compute $(\bs_0 , \lambda_0)$ by solving $Q_0$, then compute $\rho_0$ using the definition in \eqref{eq:rho-k} \label{compute_s0}.
			
			\For{$k=0, 1, 2, \ldots$}
			
			\If{$\rho_k \geq \rho$ and either $\lambda_k/\|\bs_k\|_2 \leq  \sigma_k$ or $\|\bs_k\|_2=\Delta_k$} \label{Cond-acc} \hspace{0.7 cm}(Acceptance) 
			
			\State set $\bx_{k+1} \gets \bx_k + \bs_k$ \label{x-acc}
			
			\State set $\Delta_{k+1} \gets \max\{\Delta_k, \, \gamma_E\|\bs_k\|_2\}$ \label{Delta-acc}
			
			\State set $\delta_{k+1} \gets \min\{\Delta_{k+1}, \max\{\delta_k, \, \gamma_E \|\bs_k\|_2\}\}$ \label{delta-acc}
			
			\State set $\sigma_{k+1} \gets \max\{\sigma_k , \, \lambda_k/\|\bs_k\|_2 \}$ \label{sigma-acc}
			
			\ElsIf{$\rho_k<\rho$} \label{Cond-cont}\hspace{6. cm}(Contraction) 
			
			\State set $\bx_{k+1} \gets \bx_k$ \label{x-cont}
			
			\State set $\Delta_{k+1} \gets \Delta_k$ \label{Delta-cont}
			
			\State set $\delta_{k+1} \gets$ CONTRACT$(\bx_k, \delta_k, \sigma_k, \bs_k, \lambda_k)$ defined in Algorithm \eqref{alg-contract}\label{delta-cont}
			
			\ElsIf{$\rho_k \geq \rho$, $\lambda_k/\|\bs_k\|_2 > \sigma_k$, and $\|\bs_k\|_2<\Delta_k$} \label{Cond-exp} \hspace{1 cm}(Expansion) 
			
			\State set $\bx_{k+1} \gets \bx_k$ \label{x-exp}
			
			\State set $\Delta_{k+1} \gets \Delta_k$ \label{Delta-exp}
			
			\State set $\delta_{k+1} \gets \min \{\Delta_{k+1} ,\, \lambda_k /\sigma_k \}$ \label{delta-exp}
			
			\State set $\sigma_{k+1} \gets \sigma_k$ \label{sigma-exp}
			
			\EndIf
			
			\State Compute $(\bs_{k+1}, \lambda_{k+1})$ by solving $Q_{k+1}$, then compute $\rho_{k+1}$ using \eqref{eq:rho-k} \label{compute-sk}
			
			\If{$\rho_k < \rho$} \label{Cond-rho_k}
			
			\State set $\sigma_{k+1} \gets \max\{\sigma_k, \, \lambda_{k+1}/\|\bs_{k+1}\|_2 \}$ \label{Sigma-rho_k}
			
			\EndIf
			
			\EndFor
			\EndProcedure
		\end{algorithmic}
	\end{algorithm}


	\begin{algorithm}
		\caption{CONTRACT Sub-routine}\label{alg-contract}
		\textbf{Require:} update constant $\gamma_C \in (0,1)$.\\
		\textbf{Require:} ratio bound constants $\underline{\sigma}$ and $\overline{\sigma}$ defined in \eqref{sigma-min-max}.
		
		\hrulefill
		
		\begin{algorithmic}[1]
			
			\Procedure{CONTRACT$(\bx_k, \delta_k, \sigma_k, \bs_k, \lambda_k)$}{}
			\State set $\bar{\lambda} \gets \lambda_k + \underline{\sigma}\Delta_k$  and set $\bar{\bs}$ as the solution of $Q_k(\bar{\lambda})$.\label{lambda-update-sigma}
			\If{$\|\bar{\bs}\|_2 < \|\bs_k\|_2$ and $\lambda_k < \underline{\sigma}\|\bs_k\|_2$} \label{Cond-lower-bound-no}

			\State set $\lambda \gets \bar{\lambda} +  H_{max} +  \big(\underline{\sigma}{\cal X}_k\big)^{1/2}$ and set $\bs$ as the solution of $Q_k(\lambda)$. \label{lambda-update-H}
			
			\If{$\lambda/\|\bs\|_2 \leq \overline{\sigma}$} \label{Cond-upper-bound-yes}
			
			\State \textbf{return} $\delta_{k+1} \gets \|\bs\|_2$ \label{return-1}
			
			\Else \label{Cond-upper-bound-no} 
			
			\State set $\lambda \gets \bar{\lambda}$ \label{lambda-update-2}
			
			\State \textbf{return} $\delta_{k+1} \gets \|\bar{\bs}\|_2$ \label{return-2}
			
			\EndIf 
			
			\Else \label{Cond-lower-bound-yes}
			
			\If{$\|\bar{\bs}\|_2 = \|\bs_k\|_2$} \label{update-lambda-no-update-s}
			
			\State set $\lambda \gets \gamma_{\lambda}\bar{\lambda}$ and set $\bs$ as the solution of $Q_k(\lambda)$ \label{lambda-bar-update-linear}
			
			\Else \label{update-lambda-update-s}
			
			\State set $\lambda \gets \gamma_{\lambda}\lambda$ and set $\bs$ as the solution of $Q_k(\lambda)$ \label{lambda-update-linear}
			
			\EndIf

			\While{$\|\bs\|_2 = \|\bs_k\|_2$} \label{while-loop}
			
			\State $\lambda \gets \gamma_{\lambda}\lambda$ and set $\bs$ as the solution of $Q_k(\lambda)$ \label{lambda-update-linear-while}
			
			\EndWhile
			
			\If{ $\|\bs\|_2 \geq \gamma_C \|\bs_k\|_2$} \label{Cond-s-large}
			
			\State \textbf{return} $\delta_{k+1} \gets \|\bs\|_2$ \label{return-3}
			
			\Else \label{Cond-s-small}
			
			\State \textbf{return} $\delta_{k+1} \gets \gamma_C \|\bs_k\|_2$ \label{return-4}
			\EndIf
			\EndIf
			\EndProcedure
		\end{algorithmic}
	\end{algorithm}

	\subsection{Convergence of First-Order-LC-TRACE to First-order Stationarity}
	Throughout this section, we make the following assumptions that are standard for global convergence theory of trust region methods.
	\begin{assumption}\label{Assumption1-LC}
		The objective function $f$ is twice continuously differentiable and bounded below by a scalar $f_{min}$ on ${\cal P}$. We assume that the functions $\bg(\cdot) \triangleq \nabla f(\cdot)$ and $\bH(\cdot) \triangleq \nabla^2 f(\cdot)$  are Lipschitz continuous on the path defined by the iterates computed in the Algorithm, with Lipschitz constants $L$ and $H_{Lip}$, respectively. Furthermore, we assume the gradient sequence $\{\bg_k \}$ is bounded in norm, that is, there exists a scalar constant $g_{max} > 0$ such that $\|\bg_k\|_2 \triangleq \|\nabla f(\bx_k)\|_2 \leq g_{max}$ for all $k \in \mathbb{N}$. Moreover, we assume that the Hessian sequence $\{\bH_k\}$ is bounded in norm, that is, there exist a scalar constant $H_{max}>0$, such that  $\|H_k\|_2 \triangleq \|\nabla^2 f(\bx_k)\|_2 \leq H_{max}$ for all $k \in \mathbb{N}$.
	\end{assumption}

	We next state the main results for convergence of Frist-Order-LC-TRACE.
	
	\begin{theorem}\label{first-order-convergence}
		Under Assumption~\ref{Assumption1-LC}, any limit point of the iterates generated by First-Order-LC-TRACE algorithm is a first-order stationary point .
	\end{theorem}
	\begin{proof}
		The proof of the Theorem is relegated to Appendix~\ref{first-order-convergence-app}.
	\end{proof}
	
	Unfortunately, Assumption~\ref{Assumption1-LC} is not sufficient to obtain the desired rate of convergence in the presence of constraints; in particular,
	Assumption \ref{Assumption1-LC} may not ensure a model decrease of the form
	\begin{equation}\label{eq:Model_Suff_Decrease}
	f_k - q_k(\bs_k) = -\bg_k^T\bs_k - \dfrac{1}{2}\bs_k^T\bH_k \bs_k \geq \kappa\lambda_k \|\bs_k\|_2^2,
	\end{equation}
	for some constant $\kappa \in (0,1)$. To understand this, let us first review the same result for the unconstrained scenario: it is known that $\bH_k + \lambda_k\bI \succeq 0$ at every iteration \cite[Corollary~7.2.2]{conn2000trust}. Thus, by Lemma \ref{L3.3}, we get
	\begin{equation}\label{eq:easy-case}
	f_k - q_k(\bs_k) = -\bg_k^T\bs_k - \dfrac{1}{2}\bs_k^T\bH_k \bs_k \geq \dfrac{1}{2}\lambda_k \|\bs_k\|_2^2.
	\end{equation}
	
	However, in contrast to the unconstrained case, there is no guarantee that the step $\bs_k$ satisfies \eqref{eq:Model_Suff_Decrease} in the constrained scenario. More specifically, in the presence of constraints, the condition is not guaranteed when the step $\bs_k$ provides  ascent first-order direction with negative curvature. To account for this case, we assume the following assumption holds.
	
	
	\begin{assumption}\label{Assumption 3}
		If $\bg_k^T\bs_k \geq 0$ and $\bs_k^T\bH_k\bs_k \leq 0$, there exists a sequence of feasible points $\{\bx_{k,i}\}_{i=0}^{l_k}$ with $0 \leq l_k \leq \bar{l}$, $\bx_{k,0}=\bx_k$, $\bs_{k,i}=\bx_{k,i}-\bx_k$ and $\bx_{k,l_k} = \bx_k + \bs_k$ such that for $i = 1,\ldots, l_k$,
		\[\arraycolsep=1pt\def\arraystretch{1.4}
		\begin{array}{l}
		q_k(\bs_{k,i}) \leq q_k(\bs_{k,i-1});\\ 
		\bg_k^T(\bx_{k,i} - \bx_{k,i-1}) + \bs_{k,i}^T\bH_k(\bx_{k,i} - \bx_{k,i-1}) \leq -\lambda_k \bs_{k,i}^T(\bx_{k,i} - \bx_{k,i-1});\\  
		\bg_k^T(\bx_{k,i} - \bx_{k,i-1}) + \bs_{k,i-1}^T\bH_k(\bx_{k,i} - \bx_{k,i-1}) \leq -\lambda_k \bs_{k,i-1}^T(\bx_{k,i} - \bx_{k,i-1}).
		\end{array}\]
	\end{assumption}
	
	This assumption was also used in \cite{cartis2012adaptive} to show that the number of iterations required to reach an $\epsilon$-first order stationary point when adaptive ARC algorithm is used is ${\cal O}(\epsilon^{-3/2})$. As mentioned in \cite{cartis2012adaptive}, this assumptions holds if $\bx_{k,l_k}$ is the first minimizer of the model $q_{\lambda_k}$ along the piecewise linear path ${\cal P}_k \triangleq \bigcup\limits_{i=1}^{l_k}[\bx_{k,i-1}, \bx_i].$
	Using Assumption~\ref{Assumption 3}, we obtain the desired model decrease \eqref{eq:Model_Suff_Decrease} and we have the following Theorem. 
	
	\begin{theorem}\label{epsilon-first-order-convergence-complexity}
		Under Assumptions \ref{Assumption1-LC} and \ref{Assumption 3}, for any given scalar $\epsilon \in (0,\infty)$, the total number of sub-problem routines of First-Order-LC-TRACE required to reach an $\epsilon$-first order stationary point of~\eqref{Linearly-Cons-Prob} is ${{\cal O}}(\epsilon^{-3/2} \log^3(1/\epsilon))$.
	\end{theorem}
	
	\begin{proof}
		The proof of the Theorem is relegated to Appendix~\ref{epsilon-first-order-convergence-complexity-app}.
	\end{proof}
	
	In the next section, we use this first order result to develop an algorithm for finding second order stationary points.
	
	\vspace{0.3cm}
	
	\section{Second-Order-LC-TRACE Algorithm}
	Leveraging the convergence result of First-Order-LC-TRACE, we propose algorithm~\ref{alg-GLC-cap} for converging to second order stationary points.

	\begin{algorithm}[H]
		\caption{Second-Order-LC-TRACE}\label{alg-GLC-cap}
		\textbf{Require:} The constants $\tilde{L}\triangleq \max\{L, g_{max}\}$, $\tilde{H} \triangleq \max\{H_{Lip}, H_{max}\}$, $\epsilon_g >0$, $\epsilon_H >0$.
		
		\hrulefill
		
		\begin{algorithmic}[1] \label{alg-GLC}
			\Procedure{}{}
			\State Choose a feasible point $\bx_0$.
			\State Compute $\cX_0$ and $\cpsi_0$ by solving  (\ref{eq:X_k}) and (\ref{eq:psi_k}), respectively.
			
			\For{$k=0, 1, 2, \ldots$} 
			
			\If{$\cX_k > \epsilon_g$}\label{alg-first-order-not-reached}
			\State Compute $\bx_{k+1}$ by running one iteration of First-Order-LC-TRACE starting with $\bx_k$.\label{alg-use-LCTRACE}
			\Else \label{alg-first-order-reached} 
			\State Compute $\widehat{\bd}_k$ and $\mathcal{\psi}_k$ by solving \eqref{eq:psi_k}.\label{alg-compute-d}
			\State set $\bx_{k+1} \gets \bx_k + \dfrac{2\cpsi_k}{\tilde{H}}\widehat{\bd}_k$. \label{alg-second-order}
			\EndIf
			\EndFor
			\EndProcedure
		\end{algorithmic}
	\end{algorithm}

	We now show that this algorithm can find an $(\epsilon_g, \epsilon_H)$-second-order stationary point of problem~\eqref{Linearly-Cons-Prob}.
	
	\begin{theorem}\label{epsilon-second-order-convergence-complexity}
		Under Assumptions \ref{Assumption1-LC} and \ref{Assumption 3}, for any given scalars $\epsilon_g>0$ and $\epsilon_H>0$, the total number of iterations required to reach an $(\epsilon_g, \epsilon_H)$-second-order stationary point of~\eqref{Linearly-Cons-Prob} when running Algorithm~\ref{alg-GLC-cap} is ${{\cal O}}\Big(\log^3(\epsilon_g^{-1})\max\big\{\epsilon_g^{-3/2} , \epsilon_H^{-3}\big\}\Big)$.
	\end{theorem}
	
	\begin{proof}
		The proof of the Theorem is relegated to Appendix~\ref{epsilon-second-order-convergence-complexity-app}.
	\end{proof}

	\newpage
	{
		\bibliography{references}
		\bibliographystyle{abbrv}
	}
	\newpage

	\appendix
	
	\section{Proofs for Section~\ref{sec:LC-TRACE}}
	Consider the following optimization problem
	\begin{equation}
	\displaystyle{ \operatornamewithlimits{\mbox{minimize}}_{\bx \in {\cal P}}} \, \, f(\bx),
	\end{equation}
	where ${\cal P} \triangleq \{\bx \in \mathbb{R}^n \, | \, \bA\bx \leq \bb \}$ is a polyhedron with finite number of linear constraints. In this section we generalize results from \cite{curtis2017trust} to adapt for the linear constraints. For the sake of completeness of the manuscript, some Lemmas and proofs are restated from \cite{curtis2017trust}.
	
	\vspace{0.3cm}
	
	Recall the sub-problem $Q_k$ with trust region $\delta_k$,
	\[Q_k \triangleq \displaystyle{ \operatornamewithlimits{\mbox{min}}_{\bs }} \, \, q_k(\bs) \triangleq f_k + \bg_k^T\bs + \dfrac{1}{2}\bs^T\bH_k \bs, \quad \mbox{subject to } \begin{cases} 
	\bA\bs \leq \bb - \bA\bx_k\\
	\|\bs\|_2 \leq \delta_k
	\end{cases}.
	\]
	Let $\bm{\lambda}_k^C$ be the multiplier corresponding to the linear constraint $\bA \bs \leq \bb - \bA \bx_k$, and $\lambda_k$ be the multiplier for the trust region constraint $\|\bs\|_2 \leq \delta_k$. The first order K.K.T optimality conditions for the above problem are state below \cite{bertsekas1999nonlinear}
	\begin{flalign}
	&\bg_k + (\bH_k + \lambda_k \bI)\bs_k + \bA^T \blambda_k^C = \bzero, \label{CC1}\\
	&\bzero \leq \blambda_k^C \perp \bb - \bA\bx_k - \bA \bs_k \geq \bzero, \label{CC2}\\
	&0 \leq \lambda_k \perp \delta_k - \|\bs_k\|_2^2 \geq 0. \label{CC3}
	\end{flalign}
	
	\subsection{Proof of Theorem \ref{first-order-convergence}}\label{first-order-convergence-app}
	To show convergence to first-order stationarity, we first provide in  Lemma~\ref{L3.3} a sufficient decrease condition. Then, in Lemma \ref{L3.10} we show that the number of accepted steps $|{\cal A}|$ is infinite. Combining these two results with the assumption that $f$ is lower bounded, we get the desired convergence result. In practice, the algorithm terminates when ${\cal X}_k$ is below a prescribed positive threshold $\epsilon >0$. Hence, we assume, without loss of generality that ${\cal X}_k \geq \epsilon$ for all $k\in \mathbb{N}$.
	
	\begin{lemma}\label{L3.3}
		For any $k \in \mathbb{N}$, the trial step $\bs_k$ and dual variable $\lambda_k$ satisfy
		
		\begin{equation}\label{3.2}
		f_k - q_k (\bs_k) \geq \dfrac{1}{2} \bs_k^T(\bH_k + \lambda_k \bI)\bs_k +\dfrac{1}{2}\lambda_k \|\bs_k\|_2^2.
		\end{equation} 
		
		In addition, for any $k \in \mathbb{N}_{+}$, the trial step $\bs_k$ satisfies
		\begin{equation}\label{3.3}
		f_k - q_k(\bs_k)  \geq C{\cal X}_k \min \Big\{ \delta_k, \, \dfrac{{\cal X}_k}{\|\bH_k\|_2}, 1\Big\}.
		\end{equation}
	\end{lemma}
	
	\begin{proof}
		By definition of $q_k$,
		\begin{flalign}\label{eq: proof 3.3}
		f_k - q_k(\bs_k) &= -\bg_k^T\bs_k - \dfrac{1}{2} \bs_k^T \bH_k\bs_k \nonumber\\
		& = \bs_k^T\bH_k\bs_k + \lambda_k\|\bs_k\|_2^2 + \bs_k^T\bA^T(\blambda_k^C) - \dfrac{1}{2}\bs_k^T\bH_k\bs_k \nonumber\\
		& = \dfrac{1}{2} \bs_k^T(\bH_k + \lambda_k\bI)\bs_k + \dfrac{1}{2}\lambda_k\|\bs_k\|^2 + \bs_k^T\bA^T(\blambda_k^C) \nonumber \\
		& \geq \dfrac{1}{2} \bs_k^T(\bH_k + \lambda_k\bI)\bs_k + \dfrac{1}{2}\lambda_k\|\bs_k\|^2,
		\end{flalign}
		where the second equality follows by KKT condition \eqref{CC1}, and the last inequality follows from the feasibility of $\bx_k$ and the complementary slackness \eqref{CC2}. 
		
		Also, using \cite[Theorem~12.2.2]{conn2000trust}, we obtain
		\[f_k - q_k(\bs_k)  \geq C{\cal X}_k \min \Big\{ \delta_k, \, \dfrac{{\cal X}_k}{\|\bH_k\|_2}, 1\Big\}.\]
	\end{proof}
	
	To prove the infinite cardinality of the set ${\cal A}$, we need some intermediate Lemmas. The next result shows that the trust region radius is reduced when the CONTRACT subroutine is called.

	\begin{lemma}\label{L3.4}
		For any $k \in \mathbb{N}$, if $k \in {\cal C}$, then $\delta_{k+1} < \delta_k$.
	\end{lemma}
	\begin{proof}
		Suppose that $k \in {\cal C}$. We prove the result by considering the various cases that may occur within the CONTRACT subroutine. If Step~\ref{return-4} is reached, the subroutine returns $\delta_{k+1} = \gamma_C\|\bs_k\|_2 < \delta_k$. Otherwise, if Step~\ref{return-1} is reached, the subroutine returns $\delta_{k+1} = \|\bs\|_2$ where $\bs$ solves $Q_k(\lambda)$ for $\lambda \geq \bar{\lambda}$. Hence, 
		\[
		\delta_{k+1} = \|\bs\|_2  < \|\bs_k\|_2 \leq \delta_k,
		\]
		where the strict inequality follows from Step~\ref{Cond-lower-bound-no}. Similarly, if Step~\ref{return-2} is reached, the subroutine returns $\delta_{k+1} = \|\bar\bs\|_2$ where $\bar\bs$ solves $Q_k(\bar\lambda)$. Hence, 
		\[
		\delta_{k+1} = \|\bar\bs\|_2  < \|\bs_k\|_2 \leq \delta_k.
		\]
		Otherwise, Step~\ref{return-3} is reached. In which case, the subroutine returns $\delta_{k+1} = \|\bs\|_2$ where $\bs$ solves $Q_k(\lambda)$ for $\lambda > \lambda_k$. The result follows using the while loop condition Step~\ref{while-loop} along with the inverse relationship of $\lambda$ and $\|\bs\|$.
	\end{proof}
	
	We now show that for all iterations $k$, the trust region region radius $\delta_k$ is upper bounded by a non-decreasing sequence $\{\Delta_k\}$. Also, if $k \in {\cal A} \cup {\cal E}$, we show that $\delta_{k+1}\geq \delta_k$.

	\begin{lemma}\label{L3.5-3.6}
		For any $k \in \mathbb{N}$, there holds $\delta_k \leq \Delta_k \leq \Delta_{k+1}$. Moreover, $\delta_{k+1}\geq \delta_k$ for all $k \in {\cal A} \cup {\cal E}$.
	\end{lemma}
	\begin{proof}
		The fact that $\Delta_k \leq \Delta_{k+1}$ for all $k \in \mathbb{N}$ follows from the computations in Steps \ref{Delta-acc}, \ref{Delta-cont}, and \ref{Delta-exp} of Algorithm~\ref{alg-LC}. It remains to show that $\delta_k \leq  \Delta_k$ for all $k \in \mathbb{N}$. We prove the result by means of induction. 
		
		The inequality holds for $k = 0$ by the initialization of quantities in Step~\ref{initial} of Algorithm~\ref{alg-LC}. Assume the induction hypothesis holds for iteration $k$. By the computations in Steps~\ref{Delta-acc}, \ref{delta-acc}, \ref{Delta-exp}, \ref{delta-exp} and by Lemma \ref{L3.4}, the result holds for iteration~$k+1$. We next show that $\delta_{k+1}\geq \delta_k$ for all $k \in {\cal A} \cup {\cal E}$.
		
		Suppose $k \in {\cal A}$. It follows from Steps~\ref{Delta-acc} and \ref{delta-acc} that
		\[\delta_{k+1} =  \min\{\max\{\Delta_k,\gamma_E\|\bs_k\|_2\} , \max\{\delta_k, \gamma_E\|\bs_k\|_2\}\} \geq \delta_k.\]
		Here the inequality follows since $\delta_k \leq \Delta_k \leq \Delta_{k+1}$. Now suppose $k \in {\cal E}$. By the conditions indicated in Step~\ref{Cond-exp}, we have $\lambda_k > \sigma_k\|\bs_k\|_2 \geq 0$. It follows by \eqref{CC3} that $\|\bs_k\|_2 =\delta_k$. We obtain
		\[\delta_{k+1} = \min\{\Delta_{k+1}, \lambda_k/\sigma_k\} \geq \min\{\delta_k , \|\bs_k\|_2\} = \delta_k,\]
		where the inequality follows since $\delta_k \leq \Delta_k \leq \Delta_{k+1}$.
	\end{proof}
	The next result, shows that we cannot have two consecutive expansion steps.
	
	\begin{lemma}\label{L3.7}
		For any $k \in \mathbb{N}$, if $k \in {\cal C} \cup {\cal E}$, then $k+1  \notin {\cal E}$.
	\end{lemma}

	\begin{proof}
		Observe that if $\lambda_{k+1}= 0$, then conditions in Steps~\ref{Cond-exp} of Algorithm \ref{alg-LC} ensure that
		$(k+1) \notin {\cal E}$. Thus, by \eqref{CC3}, we may proceed under the assumptions that  $\|\bs_{k+1}\|_2 = \delta_{k+1}$ and $\lambda_{k+1}> 0$.

		Suppose that $k \in {\cal C}$, i.e. $\rho_k< \rho$. It follows that Step~\ref{Sigma-rho_k} sets  $\sigma_{k+1} \geq \lambda_{k+1}/\|\bs_{k+1}\|_2$. Therefore, if $\rho_{k+1}\geq \rho$, we have $(k+1) \in {\cal A}$. Otherwise, $\rho_{k+1} <\rho,$ which implies that $(k + 1) \in {\cal C}$.
		
		Now suppose that $k \in {\cal E}$. It follows that 
		\begin{equation}\label{eq:proofL3.7}
		\lambda_k > \sigma_k \|\bs_k\|_2, \quad \delta_{k+1} = \min\{\Delta_k, \lambda_k/\sigma_k\}, \quad  \mbox{and} \quad \sigma_{k+1} = \sigma_{k}.
		\end{equation}
		
		Combined with \eqref{CC3}, we get $\|\bs_k\|_2= \delta_k$. We now consider two different cases:
		\begin{enumerate}
			\item Suppose $\Delta_k \geq  \lambda_k/\sigma_k$. It follows from \eqref{eq:proofL3.7} that 
			\begin{equation}\label{eq:proof3.7-2}
			\delta_{k+1} = \lambda_k/\sigma_k > \|\bs_k\|_2 = \delta_k.
			\end{equation}
			Therefore, by the relationship between the trust region radius and its corresponding multiplier, we get $\lambda_{k+1} \leq \lambda_k$. Combined with \eqref{eq:proofL3.7} and \eqref{eq:proof3.7-2}, we obtain
			\[\lambda_{k+1} \leq \lambda_k = \sigma_k\delta_{k+1} = \sigma_{k+1}\|\bs_{k+1}\|_2.\]
			Hence $(k + 1) \notin {\cal E}$.
			
			\item Suppose $\Delta_k < \lambda_k / \sigma_k$. Using \eqref{eq:proofL3.7}
			\[\|\bs_{k+1}\|_2= \delta_{k+1} = \Delta_k = \Delta_{k+1},\]
			where the last equality holds by Step~\ref{Delta-exp}.
			If $\rho_{k+1} \geq \rho$, then $(k + 1) \in {\cal A}_{\Delta}\subseteq {\cal A}$. Otherwise, $\rho_{k+1} < \rho$, from which it follows that $(k + 1) \in {\cal C}$.  
		\end{enumerate}
		Hence, in both cases $(k + 1) \notin {\cal E}$.
	\end{proof}
	
	Next, we show that if the dual variable for the trust region constraint $\lambda_k$ is sufficiently large, then the constraint is active and the sufficient decrease criteria is met.
	
	\begin{lemma}\label{L3.8}
		For any $k \in \mathbb{N}$, if the trial step $\bs_k$ and dual variable $\lambda_k$ satisfy
		\begin{equation}\label{3.6}
		\lambda_k \geq g_{Lip} + H_{max} + \rho \|\bs_k\|_2,
		\end{equation}
		then $\|\bs_k\|_2 = \delta_k$ and $\rho_k \geq \rho$.
	\end{lemma}
	\begin{proof}
		By the definition of the objective function of the model $q_k$, there exists a point $\bar{\bx}_k \in \mathbb{R}^n$ on the line segment $[\bx_k , \bx_k + \bs_k ]$ such that
		\begin{equation}\label{3.7}
		\begin{array}{l l}
		q_k (\bs_k ) - f (\bx_k + \bs_k) = & \big(\bg_k - \bg(\bar{\bx}_k)\big)^T \bs_k + \dfrac{1}{2} \bs_k^T \bH_k \bs_k\\
		& \geq -\|\bg_k - \bg(\bar{\bx}_k)\|_2\|\bs_k\|_2 - \dfrac{1}{2}\|\bH_k\|_2\|\bs_k\|_2^2. 
		\end{array}
		\end{equation}
		Therefore,
		\begin{flalign*}
		f_k - f(\bx_k + \bs_k) &= f_k - q_k(\bs_k)  + q_k(\bs_k) - f(\bx_k + \bs_k)\\
		&\geq \dfrac{1}{2}\bs_k^T\bH_k \bs_k + \lambda_k \|\bs_k\|_2^2 - \|\bg_k - \bg(\bar{\bx}_k)\|_2\|\bs_k\|_2  - \dfrac{1}{2}\|\bH_k\|_2\|\bs_k\|_2^2\\
		&\geq -\|\bH_k\|_2\|\bs_k\|_2^2 + \lambda_k \|\bs_k\|_2^2 - g_{Lip}\|\bs_k\|_2^2\\
		&\geq (\lambda_k  - g_{Lip} - H_{max})\|\bs_k\|_2^2\\
		&\geq \rho \|\bs_k\|_2^3.
		\end{flalign*} 
		Here the first inequality holds  from Lemma~\ref{L3.3} and expression \eqref{3.7}. The result $\|\bs_k\|=\delta_k$ follows directly from \eqref{3.6} and \eqref{CC3}.
	\end{proof}
	
	We now use the previous results to show that if from some iteration onward, all the steps are contraction steps, then the sequence of trust region radii converge to zero, and the sequence of dual variables converge to infinity.
	
	\begin{lemma}\label{L3.9}
		If $k \in {\cal C}$ for all  $k \geq k_0$, then $\{\delta_k\} \rightarrow 0$ and $\{\lambda_k\} \rightarrow \infty$.
	\end{lemma}
	
	\begin{proof}
		Assume, without loss of generality, that $k \in {\cal C}$ for all $k \in \mathbb{N}$. It follows from Lemma \ref{L3.4} that $\{\delta_k\}$ is monotonically strictly decreasing. Combined with the fact that $\{\delta_k\}$ is bounded below by zero, we have that $\{\delta_k\}$ converges. We may now observe that if Step~\ref{return-4} of the CONTRACT subroutine is reached infinitely often, then clearly,
		$\{\delta_k\} \rightarrow 0$. Hence, it follows by the relationship between the trust region radius and its corresponding multiplier that $\{\lambda_k\} \rightarrow \infty$. Therefore, let us assume that  Step~\ref{return-4} of the CONTRACT subroutine    does not occur infinitely often, i.e., that there exists $k_{\cal C} \in \mathbb{N}$ such that Step~\ref{return-1}, \ref{return-2}, or \ref{return-3} is reached for all $k \geq k_{\cal C}$. 
		Consider iteration $k_{\cal C}$. Steps \ref{lambda-update-sigma}, \ref{lambda-update-H}, \ref{lambda-bar-update-linear}, \ref{lambda-update-linear}, \ref{lambda-update-linear-while}  in the CONTRACT subroutine will set 
		\[\lambda_{k+1} =\lambda \geq \min\{ \lambda_k + \underline{\sigma}\Delta_k, \gamma_{\lambda} \lambda_k\} >\lambda_k \quad \mbox{ for all } k \geq k_{\cal C} +1.\]
		Therefore, since $k \in {\cal C}$ for all $k \geq k_{\cal C}$, we have $\bx_k = \bx_{k_{\cal C}}$ (and so ${\cal X}_k = {\cal X}_{k_{\cal C}}$) for all $k \geq k_{\cal C}$, which implies that $\{\lambda_k\}\rightarrow \infty$. It follows by the relationship between the trust region radius and its corresponding multiplier that $\|\bs_k\|_2 = \delta_k \rightarrow 0$.
	\end{proof}
	
	We now prove that the set of accepted steps is infinite.
	
	\begin{lemma}\label{L3.10}
		The set ${\cal A}$ has infinite cardinality.
	\end{lemma}
	\begin{proof}
		To derive a contradiction, suppose that $|{\cal A}| < \infty$. We claim that this implies $|{\cal C}| = \infty$. Indeed, if $|{\cal C}| <\infty$, then there exist some $k_{\cal E} \in \mathbb{N}$ such that $k \in {\cal E}$ for all $k \geq k_{\cal E}$ , which contradicts Lemma \ref{L3.7}. Thus, $|{\cal C}|= \infty$. Combining this with the result of Lemma \ref{L3.7}, we conclude that there exists some $k_{\cal C} \in \mathbb{N}_{+}$ such that $k \in {\cal C}$ for all $k \geq k_{\cal C}$. It follows from Lemma \ref{L3.9} that $\{\|\bs_k\|_2\} \leq  \{\delta_k\} \rightarrow 0$ and $\{\lambda_k\} \rightarrow \infty$. In combination with Lemma \ref{L3.8}, we conclude that there exists some $k \geq k_{\cal C}$ such that $\rho_k \geq \rho$, which contradicts the fact that $k \in {\cal C}$ for all $k \geq k_{\cal C}$. Having arrived at a contradiction under the supposition that $|{\cal A}| <\infty$, the result follows.
	\end{proof}
	
	We now provide an upper bound for the sequence $\{\Delta_k\}$ and the trial steps $\{\bs_k\}$. Moreover, we show that the number of ${\cal A}_{\Delta}$ steps computed by the algorithm is finite.
	
	\begin{lemma}\label{L3.11}
		There exists a scalar constant $\Delta>0$ and $k_{\mathcal{A}}\in \mathbb{N}$, such that $\Delta_k = \Delta$ for all  $k \geq k_{\mathcal{A}}$. Moreover, the set ${\cal A}_{\Delta}$ has finite cardinality, and there exists a scalar
		constant $s_{max}>0$ such that $\|\bs_k\|_2 \leq s_{max}$ for all $k \in \mathbb{N}$.
	\end{lemma}
	\begin{proof}
		For all $k \in {\cal A}$, we have $\rho_k \geq \rho$, which implies by Step~\ref{x-acc} of Algorithm~\ref{alg-LC} that
		\[f(\bx_k) - f(\bx_{k+1}) \geq \rho\|\bs_k\|_2^3.\]
		Combining this with Lemma \ref{L3.10} and the fact that $f$ is bounded below, it follows that $\{\bs_k\}_{k\in {\cal A}} \rightarrow 0$. In particular, there
		exists $k_{\cal A} \in \mathbb{N}$ such that for all $k \in {\cal A}$ with $k \geq k_{\cal A}$, we have
		\begin{equation}\label{3.8}
		\gamma_E\|\bs_k\|_2 \leq \Delta_0 \leq \Delta_k,  
		\end{equation}
		where the latter inequality follows from Lemma \ref{L3.5-3.6}. Combined with the update in Steps \ref{Delta-acc}, \ref{Delta-cont} and \ref{Delta-exp} of LC-TRACE, we get
		\[\Delta_{k+1} = \Delta_k \mbox{ for all } k \geq k_{\cal A}.\]
		This proves the first part of the lemma. The second part also follows from \eqref{3.8} which implies that $\|\bs_k\|_2 < \Delta_k$ for all $k \in {\cal A}$ with $k \geq k_{\cal A}$. Finally, the last part of the lemma follows from the first part and the fact that Lemma \ref{L3.5-3.6} ensures $\|\bs_k\|_2  \leq  \delta_k \leq  \Delta_k = \Delta$ for all sufficiently large $k \in \mathbb{N}$.
	\end{proof}
	
	We now show that there exits a uniform upper bound on the term $\gc$. 
	\begin{lemma}\label{LBG}
		For all $k \in \mathbb{N}$, $\gc \leq G_{max}$, where $G_{max}>0$ is a constant scalar. Moreover,
		\[\lambda_k \leq \max\{\lambda_0, \lambda_{max}\} \quad \forall \, k \in \mathbb{N},\]
		where $\lambda_{max} \triangleq \max\{g_{Lip} + 2H_{max} + (\rho +\underline{\sigma}) \Delta  + (\underline{\sigma} g_{max})^{1/2}, \, \gamma_{\lambda}(g_{Lip} + H_{max} + \rho\Delta)\}$.
	\end{lemma}
	\begin{proof}
		By \eqref{CC1} and Lemma \ref{L3.5-3.6},  
		\[\gc = \|\bH_k \bs_k + \lambda_k \bs_k\|_2 \leq (H_{max} + \lambda_k) \delta_k \leq (H_{max} + \lambda_k) \Delta.\]
		Thus, it suffices to find a constant upper bound for $\lambda_k$ to get the desired result.
		
		If $\|\bs_{k+1}\|_2 < \delta_{k+1}$, then by \eqref{CC3}, $\lambda_{k+1} = 0$. Therefore, we may proceed under the assumption that $\|\bs_{k+1}\|_2 = \delta_{k+1}$.
		Suppose $k \in {\cal C}$, then by Lemma \ref{L3.8}, $\lambda_{k} < g_{Lip} + H_{max} + \rho \Delta$.\\
		
		If Step~\ref{Cond-lower-bound-no} in the CONTRACT subroutine tests true, we get
		\begin{equation}\label{LBGex1}
		\begin{array}{ll}
		\lambda_{k+1} & \leq \lambda_k +H_{max}  + \underline{\sigma}\Delta_k + (\underline{\sigma}{\cal X}_k)^{1/2}\\
		& \leq  g_{Lip} + 2H_{max} + (\rho +\underline{\sigma}) \Delta  + (\underline{\sigma} g_{max})^{1/2}.
		\end{array}
		\end{equation}
		
		Otherwise, if Step~\ref{Cond-lower-bound-no} tests false, we claim that
		\begin{equation}\label{LBGex2}
		\lambda_{k+1} \leq \gamma_{\lambda}(g_{Lip} + H_{max} + \rho \Delta).
		\end{equation}
		
		To show our claim, we assume the contrary, i.e. $\lambda_{k+1} > \gamma_{\lambda}(g_{Lip} + H_{max} + \rho \Delta).$ Then the condition of the while loop in Step~\ref{while-loop} of the CONTRACT subroutine tested true for some $\hat{\bs}$ being a solution of $Q_{k}(\hat{\lambda})$ for $\hat{\lambda} \geq g_{Lip} + H_{max} + \rho \Delta$. 
		
		There exist $\hat{\bx}$ on the line segment $[\bx_{k} , \bx_{k} + \hat{\bs}]$ such that
		\begin{equation}\label{eq:l3.11.3}
		q_k(\hat{\bs}) - f(\hat{\bx} + \hat{\bs}) = \big(\bg_k - \bg(\hat{\bx}_k)\big)^T\bs_k +\dfrac{1}{2}\hat{\bs}^T\bH_k\hat{\bs} \geq -g_{Lip}\|\hat{\bs}\|_2^2 - \dfrac{1}{2}H_{max}\|\hat{\bs}\|_2^2.  
		\end{equation}
		
		Therefore, 
		\begin{flalign*}
		\dfrac{f(\hat{x}) - f(\hat{\bx} + \hat{\bs})}{\|\hat{\bs}\|_2^3} &= \dfrac{f(\hat{x}) - q_k(\hat{\bs}) + q_k(\hat{\bs}) - f(\hat{\bx} + \hat{\bs})}{\|\hat{\bs}\|_2^3}\\
		&\geq \dfrac{-\|\bH_k\|_2 + 2\hat{\lambda} -2g_{Lip} - H_{max}}{2\|\hat{\bs}\|_2}\\
		&\geq \dfrac{\hat{\lambda} - g_{Lip} - H_{max}}{\|\hat{\bs}\|_2}\\
		& \geq \rho,
		\end{flalign*}
		where the first inequality holds by Lemma \ref{L3.3} and \eqref{eq:l3.11.3}. Since 
		\[\rho_{k} = \dfrac{f_k - f(\bx_k + \bs_k)}{\|\bs_k\|_2^3} < \rho,\] it follows that $\|\hat{\bs}\|_2 \neq \|\bs_k\|_2$ which contradicts the condition of the while loop in Step~\ref{while-loop} which tested true for $\hat{\bs}$ generated by solving $Q_k(\hat\lambda)$.
		
		Combining \eqref{LBGex1} and \eqref{LBGex2}, we get that for all $k \in {\cal C}$
		\begin{equation}\label{LBGex3}
		\lambda_{k+1} \leq \lambda_{max},     
		\end{equation}
		where $\lambda_{max} \triangleq \max\{g_{Lip} + 2H_{max} + (\rho +\underline{\sigma}) \Delta  + (\underline{\sigma} g_{max})^{1/2}, \, \gamma_{\lambda}(g_{Lip} + H_{max} + \rho\Delta)\}$. Now, suppose that $k \in {\cal A} \cup {\cal E}$. By Lemma \ref{L3.5-3.6}, we have $\|\bs_{k+1}\|_2 = \delta_{k+1} \geq \delta_k \geq \|\bs_k\|_2.$ Hence, by the relationship between the trust region radius and its corresponding multiplier, we obtain
		\begin{equation}\label{LBGex4}
		\lambda_{k+1} \leq \lambda_k.
		\end{equation}
		
		Let $k_{\cal C} \triangleq \min\{k \in \mathbb{N} \, |\, k \in {\cal C}\}$ be the first contract step. By \eqref{LBGex4}, $\lambda_k \leq \lambda_0$ for all $k \leq k_{\cal C}$. Moreover, using \eqref{LBGex3} and \eqref{LBGex4},
		\[\lambda_k \leq \lambda_{max} \quad \forall \, k > k_{\cal C}.\]
		Combining these results yield 
		\[\lambda_k \leq \max\{\lambda_0, \lambda_{max}\} \quad \forall \, k \in \mathbb{N},\]
		which completes the proof.
	\end{proof}
	
	Notice that in the proof of Lemma \ref{LBG}, we have shown that there exists a uniform upper bound for the dual variables $\lambda_k$. Our next result shows that the ratio $\dfrac{{\cal X}_k}{\|\bs\|_2}$ is upper bounded by $C_{min} + \lambda_k$, where $C_{min}$ is a scalar constant.
	
	\begin{lemma}\label{LXksk}
		For any $k \in \mathbb{N}$, it holds
		\begin{equation}\label{eq:X_ks_k}
		{\cal X}_k \leq (C_{min} + \lambda_k)\|\bs_k\|_2,
		\end{equation}
		where $C_{min} \triangleq H_{max}  + G_{max} + g_{max}$ is a scalar constant.
	\end{lemma}
	
	\begin{proof}
		Let $\xi_{k,1}$ be the largest singular value of $\bH_k$. For all $\bd$ satisfying $\bA\bd \leq \bb - \bA\bx_k$, we have
		\begin{flalign}\label{eq:proofL3.12}
		\bg_k^T\bd &= -\bd^T(\bH_k + \lambda_k\bI)\bs_k^T - (\blambda_k^C)^T\bA\bd \nonumber \\
		&\geq -\bd^T(\bH_k + \lambda_k\bI)\bs_k^T - (\blambda_k^C)^T\bA\bs_k \nonumber \\
		&\geq -(\xi_{k,1} + \lambda_k)\|\bd\|_2\|\bs_k\|_2 - (\blambda_k^C)^T\bA\bs_k,
		\end{flalign}
		where the first equality holds by \eqref{CC1}, and the first inequality holds by  complementary slackness \eqref{CC2}. Minimizing over all such $\bd$, we obtain
		\begin{flalign*}
		\underset{\bA\bd \leq \bb - \bA\bx_k, \, \, \|\bd\|\leq 1}{\min} \,\,\bg_k^T\bd &\geq -(\xi_{k,1} + \lambda_k)\|\bs_k\|_2-\|(\blambda_k^C)^T\bA\|_2\|\bs_k\|_2\\
		&\geq -(\xi_{k,1} + \lambda_k)\|\bs_k\|_2-(\|g_k + (\blambda_k^C)^T\bA\|_2+ \|g_k\|_2)\|\bs_k\|_2\\
		&\geq -(H_{max} + \lambda_k)\|\bs_k\|_2 - (G_{max} + g_{max})\|\bs_k\|_2,
		\end{flalign*}
		where the last inequality uses Lemma \ref{LBG}. Then definition of ${\cal X}_k$ yields
		\begin{flalign*}
		{\cal X}_k &\leq ( H_{max} + \lambda_k + G_{max} + g_{max})\|\bs_k\|_2\\
		&=(C_{min}+\lambda_k)\|\bs_k\|_2,
		\end{flalign*}
		where $C_{min} \triangleq H_{max} +  G_{max} + g_{max}$ is a scalar constant.
	\end{proof}
	
	We now show that the limit inferior of stationarity measure ${\cal X}_k$ is equal to zero.
	
	\begin{lemma}\label{L3.13}
		There holds
		\[\displaystyle{ \operatornamewithlimits{\mbox{lim inf}}_{k \in \mathbb{N}, \, k \rightarrow \infty}} {\cal X}_k =0.\]
	\end{lemma}
	
	\begin{proof}
		Suppose the contrary that there exists a scalar constant  ${\cal X}_{min}  > 0$ such that ${\cal X}_k \geq {\cal X}_{min}$ for all $k \in \mathbb{N}$. Then by Lemmas \ref{LBG} and \ref{LXksk}, for scalar
		\[s_{min}  = \dfrac{{\cal X}_{min}}{C_{min} + \max\{\lambda_{max}, \lambda_0\}}\,,\]
		we have that $\|\bs_k\|_2 \geq s_{min} > 0$ for all $k \in \mathbb{N}$. Moreover, for all $k \in {\cal A}$ we have $f_k - f_{k+1} \geq \rho\|\bs_k\|_2^3 > 0$. Given the lower boundedness of $f$ and Lemma \ref{L3.10} that insures infinite cardinality  of set ${\cal A}$,  we have $\{\bs_k \}_{k \in {\cal A}} \rightarrow 0$. This contradicts the existence of $s_{min} > 0$. 
	\end{proof}
	
	\begin{theorem}\label{L3.14}
		Under Assumption \ref{Assumption1-LC}, it holds that
		\begin{equation}\label{3.11}
		\lim_{k \in \mathbb{N}, \, k \rightarrow \infty} \, {\cal X}_k =0.
		\end{equation}
	\end{theorem}
	
	\begin{proof}
		Suppose the contrary that \eqref{3.11} does not hold. Combined with Lemmas \ref{L3.10} and \ref{L3.13}, it implies that there exist an infinite sub-sequence $\{t_i\} \subseteq {\cal A}$ (indexed over $i \in \mathbb{N}$) such that ${\cal X}_{t_i}  \geq 2\epsilon_{\cal X}$ for some $\epsilon_{\cal X} > 0$ and all $i \in \mathbb{N}$. Additionally, Lemmas \ref{L3.10} and \ref{L3.13} imply that there exist an infinite subsequence $\{l_i\} \subseteq {\cal A}$ such that
		\begin{equation}\label{3.12}
		{\cal X}_k \geq \epsilon_{\cal X} \mbox{ and } {\cal X}_{l_i}  < \epsilon_{\cal X} \quad \forall \, \, i\in \mathbb{N}, \,\, k \in \mathbb{N}, \, \, t_i \leq k <l_i.
		\end{equation}
		
		We claim that for all $k \in \mathbb{N}_{+}$, the trial step $\bs_k$ satisfies the following
		\begin{equation}\label{eq:Lproof3.14}
		\|\bs_k\|_2 \geq \min \Big\{\delta_k, \dfrac{{\cal X}_k}{C_{min}} \Big\}.
		\end{equation}
		The proof of this claim follows directly from Lemma \ref{LXksk}. If $\|\bs_k\|_2 = \delta_k$, the result trivially holds. Otherwise, using KKT condition \eqref{CC3}, $\lambda_k=0$ which proves our claim when combined with Lemma \ref{LXksk}. 
		
		We now restrict our attention to indices in the infinite index set
		\[{\cal K} \triangleq \{k \in {\cal A}: t_i \leq k < l_i \mbox{ for some } i \in \mathbb{N}\}.\]
		Observe from \eqref{3.12} and \eqref{eq:Lproof3.14} that
		\begin{equation}\label{3.13}
		f_k - f_{k+1} \geq \rho \|\bs_k\|_2^3 \geq \rho \Big(\min\Big\{\delta_k , \dfrac{\epsilon_{\cal X}}{C_{min}}\Big\}\Big)^{3}.
		\end{equation}
		
		Since $\{ f_k \}$ is monotonically decreasing and bounded below, we know that $f_k \rightarrow \underline{f}$ for some $\underline{f} \in \mathbb{R}$. When combined with \eqref{3.13}, we obtain
		\begin{equation}\label{3.14}
		\lim_{k \in {\cal K}, k \rightarrow \infty} \delta_k =0.
		\end{equation}
		Using this fact and Lemma \ref{L3.3}, we have for all sufficiently large $k \in {\cal K}$ that
		\begin{flalign*}
		f_k - f_{k+1} & =f_k - q_k(\bs_k) +q_k(\bs_k)  - f_{k+1}\\
		&\geq C{\cal X}_k \min \Big\{\delta_k, \dfrac{{\cal X}_k}{\|\bH_k\|_2}, 1 \Big\} - (g_{Lip} + \dfrac{1}{2}H_{max})\|\bs_k\|_2^2\\
		& \geq  C\epsilon_{\cal X} \min \Big\{\delta_k, \dfrac{\epsilon_{\cal X}}{H_{max}}, 1 \Big\} - (g_{Lip} + \dfrac{1}{2}H_{max})\|\bs_k\|_2^2\\
		& \geq  C\epsilon_{\cal X}\delta_k - (g_{Lip} + \dfrac{1}{2}H_{max})\delta_k^2\\
		& \geq  \dfrac{C}{2}\epsilon_{\cal X}\delta_k.
		\end{flalign*}
		Consequently, for all sufficiently large $i \in \mathbb{N}$, we have
		\begin{flalign*}
		\|\bx_{t_i}  - \bx_{l_i}\|_2 &\leq \sum_{k \in {\cal K}, k=t_i}^{l_i -1} \|\bx_k - \bx_{k+1}\|_2\\
		& \leq \sum_{k \in {\cal K}, k=t_i}^{l_i -1} \delta_k  \leq \sum_{k \in {\cal K}, k=t_i}^{l_i -1} \dfrac{2}{C\epsilon_{\cal X}} (f_k - f_{k+1}) = \dfrac{2}{C\epsilon_{\cal X}}(f_{t_i} - f_{l_i}).
		\end{flalign*}
		Since $\{ f_{t_i} -  f_{l_i}\} \rightarrow 0$, we get $\{\|\bx_{t_i} -\bx_{l_i}\|_2\}\rightarrow 0$, which, in turn, implies that $\{{\cal X}_{t_i} - {\cal X}_{l_i}\} \rightarrow 0$. This contradicts  \eqref{3.12}. 
	\end{proof}

	\subsection{Proof of Theorem \ref{epsilon-first-order-convergence-complexity}}\label{epsilon-first-order-convergence-complexity-app}
	In this section we show that the number of iterations required to reach an $\epsilon$-first order stationary point is ${\cal O}(\epsilon^{-3/2}\log^3 \epsilon^{-1})$. To that end, we start by showing the desired model decrease using Assumption \ref{Assumption 3}.
	
	\begin{lemma}\label{L4.5_ACR}
		Consider the directions $\bs$, $\bs^{+}$ and points $\bx = \bx_k + \bs$, $\bx^{+} = \bx_k + \bs^{+}$. If for some $\bar{\kappa} \in (0,1]$
		\begin{flalign}
		&f_k -q_k(\bs) \geq  \bar{\kappa} \lambda_k \|\bs\|_2^2, \label{eq:l1}\\
		&q_k(\bs^{+}) \leq q_k (\bs), \label{eq:l2}\\
		&\bg_k^T(\bs^{+} - \bs) + (\bs^{+})^T\bH_k(\bs^{+} - \bs) \leq -\lambda_k (\bs^{+})^T(\bs^{+} - \bs),  \label{eq:l3}\\
		&\bg_k^T(\bs^{+} - \bs) + \bs^T\bH_k(\bs^{+} - \bs) \leq -\lambda_k \bs^T(\bs^{+} - \bs),  \label{eq:l4}
		\end{flalign}
		then 
		\[f_k -q_k(\bs^+) \geq \dfrac{1}{3}\bar{\kappa} \lambda_k\|\bs^{+}\|_2^2.\]
		
	\end{lemma}
	
	\begin{proof}
		Suppose that for a given constant scalar $\alpha \in (0,1)$, $\|\bs\|_2 \geq \alpha \|\bs^{+}\|_2$. Then, it directly follows by~\eqref{eq:l1} and \eqref{eq:l2} that 
		\begin{equation}\label{eq:L4.5_ARC_1}
		f_k - q_k(\bs^{+}) = f_k - q_k(\bs) + q_k(\bs) - q_k(\bs^{+})\geq \bar{\kappa}\lambda_k\|\bs\|_2^2 \geq \bar{\kappa} \alpha^2 \lambda_k\|\bs^{+}\|_2^2.
		\end{equation}
		Now consider the case that $\|\bs\|_2 < \alpha \|\bs^{+}\|_2$. First note that by \eqref{eq:l3}, 
		\begin{equation}\label{eq:l4.5-1}
		\begin{array}{ll}
		0 &  \geq  (\bg_k + \bH_k \bs^{+})^T(\bs^{+} - \bs) +\lambda_k (\bs^{+})^T(\bs^{+} - \bs)\\
		& = (\bg_k + \bH_k \bs)^T(\bs^{+} - \bs)  +(\bs^{+} - \bs)^T\bH_k(\bs^{+} - \bs) +\lambda_k (\bs^{+})^T(\bs^{+} - \bs).
		\end{array}
		\end{equation}
		Also, by \eqref{eq:l4}
		\begin{equation}\label{eq:l4.5-2}
		(\bg_k + \bH_k \bs)^T(\bs^{+} - \bs) + \lambda_k\bs^T(\bs^{+} - \bs) \leq 0.
		\end{equation}
		
		Adding \eqref{eq:l4.5-1} and \eqref{eq:l4.5-2}, we get
		\begin{equation}\label{eq:l11}
		\begin{array}{ll}
		q_k(\bs^{+}) - q_k(\bs)  &= (\bg_k + \bH_k \bs)^T(\bs^{+} - \bs) + \dfrac{1}{2}(\bs^{+} - \bs)^T\bH_k (\bs^{+} - \bs)\\
		&\leq -\dfrac{1}{2}\lambda_k (\|\bs^{+}\|_2^2 - \|\bs\|_2^2).
		\end{array}
		\end{equation}
		Since $\|\bs\|_2 < \alpha \|\bs^{+}\|_2$, it follows that
		\begin{equation}\label{eq:L4.5_ARC_2}
		f_k - q_k(\bs^{+}) \geq q_k(\bs) - q_k(\bs^{+}) \geq \dfrac{1}{2}\lambda_k (\|\bs^{+}\|_2^2 - \|\bs\|_2^2) \geq \dfrac{1}{2}\lambda_k\bar{\kappa} \|\bs^{+}\|_2^2 (1- \alpha^2),
		\end{equation}
		where the first inequality holds by \eqref{eq:l1}, the second inequality holds by \eqref{eq:l11}, and the last inequality holds because  $\|\bs\|_2 < \alpha \|\bs^{+}\|_2$. We now choose the value of $\alpha$ for which the lower bounds~\eqref{eq:L4.5_ARC_1} and \eqref{eq:L4.5_ARC_2} are equal; i.e. $\alpha^2 = \dfrac{1}{2}\left(1 - \alpha^2\right)$, equivalently $\alpha = \sqrt{\dfrac{1}{3}}$.
	\end{proof}

	We next show that sufficient model decrease is satisfied when either $\bg_k^T\bs_k \leq 0$ or $\bs_k^T\bH_k\bs_k \geq 0$.
	
	\begin{lemma}\label{L4.4_ACR}
		Suppose that $\bg_k^T\bs_k \leq 0$ or $\bs_k^T\bH_k\bs_k\geq 0$. Then,
		\begin{equation}
		f_k - q_k(\bs_k) = -\bg_k^T\bs_k - \dfrac{1}{2}\bs_k^T\bH_k\bs_k \geq \dfrac{1}{2}\lambda_k\|\bs_k\|_2^2.
		\end{equation}
	\end{lemma}
	
	\begin{proof}
		First notice that since the origin is feasible in $Q_k$, 
		\[\bg_k^T\bs_k + \dfrac{1}{2}\bs_k^T\bH_k\bs_k \leq 0.\]
		Hence,
		\[\bs_k^T\bH_k\bs_k\geq 0 \Rightarrow \bg_k^T\bs_k \leq 0.\]
		
		On the other hand, if $\bg_k^T\bs_k \leq 0$, by \eqref{CC1},
		\begin{equation}\label{eq:lemma4.4-1}
		2\big(\bg_k^T \bs_k + \dfrac{1}{2}\bs_k^T\bH_k\bs_k + \dfrac{1}{2}\lambda_k\|\bs_k\|_2^2\big) = \bg_k^T\bs_k - \bs_k^T \bA^T \blambda_k^C \leq \bg_k^T\bs_k \leq 0,
		\end{equation}
		where the first inequality holds due to the complementary slackness condition~\eqref{CC2}.
	\end{proof}

	\begin{lemma}\label{L4.6_ACR}
		Suppose Assumption \ref{Assumption 3} holds at iteration $k$. Then there exist a constant $\kappa>0$ independent of $k$ such that 
		\[f_k - q_k(\bs_k) \geq \kappa \lambda_k \|\bs_k\|_2^2.\]
	\end{lemma}
	\begin{proof}
		If $\bg_k^T \bs_k \leq 0$ or $\bs_k^T\bH_k\bs_k \geq 0$, the result follows by Lemma \ref{L4.4_ACR}. Thus, we may proceed under the assumption that $\bg_k^T\bs_k \geq 0$ and $\bs_k^T\bH_k\bs_k \leq 0$. Using Assumption~\ref{Assumption 3}, we proceed with a proof by induction on $l_k$. If $l_k = 1$, the last condition in Assumption~\ref{Assumption 3} implies that $\bg_k^T \bs_k \leq 0$, thus Lemma~\ref{L4.4_ACR} implies the desired result. Now assume the result holds for $l_k = i<\bar{l}$, we next show that the result holds for $l_k = i+1$.
		By the induction step and Assumption \ref{Assumption 3}, we have
		\[f_k - q_k(\bs_{k,i}) \geq \bar{\kappa} \lambda_k\|\bs_{k,i}\|_2^2,\]
		\[q_k(\bs_{k,i+1}) \leq q_k(\bs_{k,i}),\]
		\[\big\langle \bg_k + \bH_k\bs_{k,i+1} , \bx_{k,i+1} - \bx_{k,i} \big\rangle \leq  - \lambda_k\bs_{k,i+1}^T(\bx_{k,i+1} - \bx_{k,i+1}),\]
		\[ \big\langle \bg_k + \bH_k\bs_{k,i} , \bx_{k,i+1} - \bx_{k,i} \big\rangle \leq - \lambda_k\bs_{k,i}^T(\bx_{k,i+1} - \bx_{k,i+1}).\]
		Then using Lemma~\ref{L4.5_ACR}, with $\bx = \bx_{k,i}$ and $\bx^{+} = \bx_{k,i+1}$ we obtain
		\[f_k - q_k(\bs_{k,i+1}) \geq \bar{\kappa}^{+} \lambda_k \|\bs_{k,i+1}\|_2^2,\]
		for some $\bar{\kappa}^{+} \in (0,1)$ independent of $k$. 
	\end{proof}
	
	
	

	Our next result provides a bound on the ratio $\lambda_{k+1} / \|\bs_{k+1}\|_2$ when $k \in {\cal C}$. 
	
	\begin{lemma}\label{L3.17}
		Assume Assumption \ref{Assumption 3} holds at iteration $k \in {\cal C}$. Then,
		\begin{itemize}
			\item If Step~\ref{return-1}, \ref{return-2}, or \ref{return-3} of Algorithm~\ref{alg-contract} is reached, then
			\[ \underline{\sigma} \leq \dfrac{\lambda_{k+1}}{\|\bs_{k+1}\|_2}\leq \max \Big\{\overline{\sigma},\Big(\dfrac{\gamma_{\lambda}}{\gamma_{C}} \Big)\dfrac{H_{Lip} + 2 \rho}{2 \kappa}\Big\}.\]
			\item If Step~\ref{return-4} of Algorithm~\ref{alg-contract} is reached, then
			\[ \dfrac{\lambda_{k+1}}{\|\bs_{k+1}\|_2}\leq \max \Big\{\overline{\sigma},\Big(\dfrac{\gamma_{\lambda}}{\gamma_{C}} \Big)\dfrac{H_{Lip} + 2 \rho}{2 \kappa}\Big\}.\]
		\end{itemize}
	\end{lemma}
	
	\begin{proof}
		Let $k \in {\cal C}$ and consider the three possible cases. The
		first two correspond to situations in which the conditions in Step~\ref{Cond-lower-bound-no} in the CONTRACT subroutine tests true.
		\begin{itemize}
			\item Suppose that Step~\ref{return-1} is reached. Then, $\delta_{k+1}= \|\bs\|_2$ where $(\lambda, \bs)$ is computed in Step~\ref{lambda-update-H}. It follows that Step~\ref{Sigma-rho_k} in Algorithm~\ref{alg-LC} will then produce the primal-dual pair $(\bs_{k+1}, \lambda_{k+1})=(\bs, \lambda)$ with $\lambda >0$. Since the condition in Step~\ref{Cond-upper-bound-yes} tested true, we have
			\begin{equation}\label{L3.17-1}
			\underline{\sigma} \leq \dfrac{\lambda_k + \underline{\sigma}{\Delta_k}}{\Delta_k} \leq \dfrac{\lambda_{k+1}}{\|\bs_{k+1}\|_2} = \dfrac{\lambda}{\|\bs\|_2} \leq \overline{\sigma},
			\end{equation}
			where the second inequality holds since $\|\bs_{k+1}\|_2 =\delta_{k+1} \leq \|\bs_k\|_2 \leq \Delta_k$.
			
			\item Suppose that Step~\ref{return-2} is reached. Then, $\delta_{k+1} = \|\bar\bs\|_2 $ where $(\bar\lambda, \bar\bs)$ is computed in Step~\ref{lambda-update-sigma}. Similar to the previous case, it follows that Step~\ref{Sigma-rho_k} in Algorithm~\ref{alg-LC} will produce the primal-dual pair $(\bs_{k+1}, \lambda_{k+1})= (\bar\bs, \bar\lambda)$ with $\bar\lambda = \lambda_k + \underline{\sigma}\Delta_k$. We first show that $\|\bar\bs\|_2 \geq \underline{\sigma}$. Assume the contrary, then by Lemma \ref{LXksk} and the fact that ${\cal X}_{k+1} ={\cal X}_k$ for all $k \in {\cal C}$,
			\[{\cal X}_k \leq (C_{min} + max\{\lambda_{max} , \lambda_0\})\|\bar\bs\|_2 \leq (C_{min} + max\{\lambda_{max} , \lambda_0\})\underline{\sigma} = \epsilon,\]
			which contradicts our assumption on ${\cal X}_k$. Here the last inequality uses the definition \eqref{sigma-min-max} of $\underline{\sigma}$. Combined with Lemma \ref{L3.5-3.6} we obtain
			\[\underline{\sigma} \leq \|\bar\bs\|_2 \leq \|\bs_k\|_2 \leq \Delta_k.\]
			Therefore,
			\begin{equation}\label{L3.17-2}
			\underline{\sigma} \leq \dfrac{ \lambda_k + \underline{\sigma}\Delta_k}{\Delta_k} \leq \dfrac{\bar\lambda}{\|\bar\bs\|_2} \leq \dfrac{\lambda_k + \underline{\sigma}\Delta_k}{\underline{\sigma}} \leq \dfrac{\underline{\sigma}(\|\bs_k\|_2 + \Delta_k)}{\underline{\sigma}} \leq 2\Delta_k \leq \bar{\sigma},
			\end{equation}
			where the fourth inequality holds by the condition of Step~\ref{Cond-lower-bound-no} and the last inequality holds by Lemmas \ref{L3.5-3.6}, \ref{L3.11} and the definition of $\bar{\sigma}$ in \ref{sigma-min-max}.\\
			
			The other case correspond to situations in
			which the condition in Step~\ref{Cond-lower-bound-no} tests false. It follows by Steps~\ref{lambda-update-sigma} and \ref{Cond-lower-bound-no} that
			\begin{equation}\label{3.17}
			\underline{\sigma} \leq  \dfrac{\lambda}{\|\bs\|_2},
			\end{equation}
			Finally, using the argument of Lemma \ref{LBG}, we claim that 
			\begin{equation}\label{3.17-2}
			\lambda_{k+1} \leq \gamma_{\lambda}\Big(\dfrac{H_{Lip} + 2\rho}{2 \kappa}\Big)\|\bs_k\|_2.
			\end{equation}

			To show our claim, we assume the contrary, i.e. $\lambda_{k+1} > \gamma_{\lambda}\Big(\dfrac{H_{Lip} + 2\rho}{2 \kappa}\Big)\|\bs_k\|_2.$ Then the condition of the while loop in Step~\ref{while-loop} tested true for some $\hat{\bs}$ computed by solving $Q_k(\hat{\lambda})$ for $\hat{\lambda} \geq (\dfrac{H_{Lip} + 2\rho}{2 \kappa})\|\hat{\bs}\|_2$. 
			
			There exists $\hat{\bx}$ on the line segment $[\bx_{k} , \bx_{k} + \hat{\bs}]$ such that
			\begin{equation}\label{eq:l-24-1}
			q_k(\hat{\bs}) - f(\hat{\bx} + \hat{\bs}) = \dfrac{1}{2}\hat{\bs}^T\big(\bH_k - \bH(\bx_k)\big)\hat{\bs} \geq -\dfrac{1}{2}H_{Lip}\|\hat{\bs}\|_2^3.  
			\end{equation}
			
			Therefore,
			\begin{flalign*}
			\dfrac{\hat{f} - f(\hat{\bx} + \hat{\bs})}{\|\hat{\bs}\|_2^3} &= \dfrac{\hat{f} - q_k(\hat{\bs}) + q_k(\hat{\bs}) - f(\hat{\bx} + \hat{\bs})}{\|\hat{\bs}\|_2^3}\\
			&\geq \dfrac{\kappa \hat{\lambda}\|\hat{\bs}\|_2^2 - 0.5H_{Lip}\|\hat{\bs}\|_2^3}{\|\hat{\bs}\|_2^3}\\
			&\geq \dfrac{\|\bs_k\|_2\rho }{\|\hat{\bs}\|_2}\\
			& \geq \rho,
			\end{flalign*}
			where the first inequality holds by  Lemma \ref{L4.6_ACR} and \eqref{eq:l-24-1}, and the last inequality holds since $\|\bs_k\|_2 \geq \|\hat{\bs}\|_2$. However $\rho_{k} = \dfrac{f_k - f(\bx_k + \bs_k)}{\|\bs_k\|_2^3} < \rho$. It follows that $\|\hat{\bs}\|_2 \neq \|\bs_k\|_2$ which contradicts the condition of the while loop in Step~\ref{while-loop} for $\hat\bs$ computed by solving $Q_k(\hat\lambda)$.

			\item Suppose that Step~\ref{return-3} is reached. Then, $\delta_{k+1} = \|\bs\|_2$. It follows that Step~\ref{Sigma-rho_k} in Algorithm~\ref{alg-LC} will produce the primal-dual pair $(\bs_{k+1}, \lambda_{k+1})$ solving $Q_{k+1}$ such that $\bs_{k+1} = \bs$ and
			$\lambda_{k+1} = \lambda$. In conjunction with \eqref{3.17}, \eqref{3.17-2}, and the condition in Step~\ref{Cond-s-large} of the CONTRACT sub-routine, we observe that
			\begin{equation}\label{L3.17-3}
			\underline{\sigma} \leq \dfrac{\lambda_{k+1}}{\|\bs_{k+1}\|_2}= \dfrac{\lambda}{\|\bs\|_2}  \leq \Big(\dfrac{\gamma_{\lambda}}{\gamma_{C}} \Big)\dfrac{H_{Lip} + 2 \rho}{2 \kappa}
			\end{equation}
			
			\item Suppose that Step~\ref{return-4} is reached. Then, $\delta_{k+1}= \gamma_C\|\bs_k\|_2$. It follows that Step~\ref{Sigma-rho_k} in Algorithm~\ref{alg-LC} will produce the primal-dual pair $(\bs_{k+1}, \lambda_{k+1}) = (\bs, \lambda)$. If  $\|\bs\|_2 < \delta_{k+1} = \gamma_C\|\bs_k\|_2$, then $\lambda_{k+1} = 0$. Otherwise, $\|\bs\|_2 = \delta_{k+1} = \gamma_C\|\bs_k\|_2$ and $\lambda_{k+1} > 0$. Combined with \eqref{3.17} and \eqref{3.17-2},  we obtain
			
			\[ \dfrac{\lambda_{k+1}}{\|\bs_{k+1}\|_2}  = \dfrac{\lambda}{\|\bs\|_2} \leq \Big(\dfrac{\gamma_{\lambda}}{\gamma_C} \Big)\Big(\dfrac{H_{Lip} + 2\rho}{2 \kappa}\Big).\]
			
		\end{itemize}
		The result follows since we have obtained the desired inequalities in all cases.
	\end{proof}
	
	We now provide an upper bound for the sequence $\{\sigma_{max}\}$.
	\begin{lemma}\label{L3.18}
		Assume Assumption \ref{Assumption 3} holds. There exists a scalar constant $\sigma_{max} > 0$ such that
		\[\sigma_{k} \leq \sigma_{max} \quad \forall \, \,k \in \mathbb{N}.\]
	\end{lemma}
	
	\begin{proof}
		First note that by Lemma~\ref{L3.11}, the cardinality of the set ${\cal A}_{\Delta}$ is finite. Hence, there exist $k_{{\cal A}} \in \mathbb{N}$ such that $k \notin {\cal A}_{\Delta}$ for all $k \geq k_{{\cal A}}$. We continue by showing that $\sigma_k$ is upper bounded for all $k \geq k_{{\cal A}}$. Consider the following three cases:
		
		
		\begin{itemize}
			\item If $k \in {\cal A}_{\sigma}$, then by definition $\lambda_k \leq \sigma_k \|\bs_k\|_2$, which implies by Step~\ref{sigma-acc} of Algorithm~\ref{alg-LC} that 
			\[ \sigma_{k+1}= \max\{\sigma_k , \lambda_k /\|\bs_k\|_2 \} = \sigma_k.\]
			
			\item If $k \in {\cal C}$, by Step~\ref{Sigma-rho_k} of Algorithm~\ref{alg-LC} and Lemma \ref{L3.17}, it follows that
			
			\begin{equation}\label{expL3.18}
			\sigma_{k+1} =  \max \Big\{ \sigma_k , \dfrac{\lambda_{k+1}}{\|\bs_{k+1}\|_2}\Big\} \leq \max \Big\{\sigma_k, \overline{\sigma},\Big(\dfrac{\gamma_{\lambda}}{\gamma_{C}} \Big)\dfrac{H_{Lip} + 2 \rho}{2 \kappa} \Big\}.
			\end{equation}
			
			\item If $k \in {\cal E}$, then Step~\ref{sigma-exp} of Algorithm~\ref{alg-LC} implies that $\sigma_{k+1} = \sigma_k$.
		\end{itemize}
		Combining the results of these three cases, the desired result follows.
	\end{proof}
	
	We now establish an upper bound on the norm trial steps $\bs_k$ when $k \in {\cal A}_{\sigma}$.
	\begin{lemma}\label{L3.19}
		For all $k \in {\cal A}_{\sigma}$, the accepted step $\bs_k$ satisfies
		\[\|\bs_k\|_2 \geq (H_{Lip} + \sigma_{max})^{-1/2}{\cal X}_{k+1}^{1/2}.\]
	\end{lemma}
	
	\begin{proof}
		For all $k \in {\cal A}_{\sigma}$ , there exists $\bar{\bx}_k$ on the line segment $[\bx_k , \bx_k + \bs_k]$ such that
		\begin{flalign*}
		\bg_{k+1}^T\bd & = \bg_{k+1}^T\bd - \bg_k^T\bd - \bd^T(\bH_k + \lambda_k\bI)\bs_k - \bd^T \bA^T \blambda_k^C\\
		& = \big(\bg(\bx_k + \bs_k) - \bg_k \big)^T\bd - \bd^T(\bH_k + \lambda_k\bI)\bs_k - \bd^T\bA^T\blambda_k^C\\
		& = \bd^T(\bH(\bar{\bx}_k) - \bH_k)\bs_k - \lambda_k\bd^T\bs_k - \bd^T\bA^T\blambda_k^C\\
		& \geq -H_{Lip}\|\bs_k\|_2^2\|\bd\|_2 - \lambda_k\|\bs_k\|_2\|\bd\|_2 - \bd^T\bA^T\blambda_k^C\\
		& \geq -H_{Lip}\|\bs_k\|_2^2 - \sigma_k\|\bs_k\|_2^2 - \bd^T\bA^T\blambda_k^C \quad \mbox{ for all } \bd \mbox{ with } \|\bd\|_2\leq 1,
		\end{flalign*}
		where the first equation follows from \eqref{CC1} and the last inequality follows since $\lambda_k \leq \sigma_k \|\bs_k\|_2$ for all $k \in {\cal A}_{\sigma}$. Thus 
		\begin{equation}\label{eq:proofL17}
		\displaystyle{ \operatornamewithlimits{\mbox{min}}_{\mbox{s.t. } d \, \in \,  {\cal D}_{k+1}}} \, \bg_{k+1}^T\bd    \geq -H_{Lip}\|\bs_k\|_2^2 - \sigma_k\|\bs_k\|_2^2 - \displaystyle{ \operatornamewithlimits{\mbox{max}}_{\mbox{s.t. } d \, \in \,  {\cal D}_{k+1}}} \, \bd^T\bA^T\blambda_k^C, 
		\end{equation}
		where ${\cal D}_{k+1} \triangleq \{ \bd \in \mathbb{R}^n \, | \, \|\bd\|_2 \leq 1; \, \, \, \bA\bd \leq \bb - \bA\bx_{k+1}\}$. Note that since $k \in {\cal A}_{\sigma}$ the updated step will be $\bx_{k+1}= \bx_k + \bs_k$. Now, let ${\cal I}_k \triangleq \{ i \, | \, \ba_i^T\bs_k = \bb_i - \ba_i^T\bx_k \}$, then 
		\[\bA_{{\cal I}_k} \bd \leq \bb_{{\cal I}_k} - \bA_{{\cal I}_k}\bx_{k+1} = \bb_{{\cal I}_k} - \bA_{{\cal I}_k}\bx_{k} - \bA_{{\cal I}_k}\bs_{k} =\bzero \Rightarrow (\blambda_k^C)^T\bA\bd = (\blambda_k^C)_{{\cal I}_k}^T\bA_{{\cal I}_k}\bd \leq 0,\] 
		for all $\bd \in {\cal D}_{k+1}$.
		Substituting in \eqref{eq:proofL17}, we obtain
		\[{\cal X}_{k+1} \leq H_{Lip}\|\bs_k\|_2^2 + \sigma_k\|\bs_k\|_2^2,\]
		which along with Lemma \ref{L3.18}, implies the result.
	\end{proof}
	
	We are now ready to compute a worst-case upper bound on the number of steps in ${\cal A}_{\sigma}$ for which the first-order criticality measure ${\cal X}_k$ is larger than a prescribed $\epsilon >0$.
	
	\begin{lemma}\label{L3.20}
		Assume Assumption \ref{Assumption 3} holds. For a scalar $\epsilon \in (0,\infty)$, the total number of elements in the index set
		\[ {\cal K}_{\epsilon}\triangleq \{k \in \mathbb{N}_{+} : k\geq 1; \, (k-1) \in {\cal A}_{\sigma}; \, {\cal X}_k >\epsilon\}\]
		is at most
		\begin{equation}\label{3.18}
		\Bigg\lceil \Bigg( \dfrac{f_0 - f_{min}}{\rho(H_{Lip}  + \sigma_{max})^{-3/2}} \Bigg)\epsilon^{-3/2}\Bigg\rceil \triangleq K_{\sigma}(\epsilon) \geq 0. 
		\end{equation}
	\end{lemma}
	
	\begin{proof}
		By Lemma \ref{L3.19}, we have for all $k \in {\cal K}_{\epsilon}$ that
		\begin{flalign*}
		f_{k-1} - f_k &\geq \rho \|\bs_{k-1}\|_2^3\geq \rho (H_{Lip} + \sigma_{max})^{-3/2}{\cal X}_k^{3/2} \geq \rho (H_{Lip} + \sigma_{max})^{-3/2}\epsilon^{3/2}.
		\end{flalign*}
		
		In addition, we have by Theorem \ref{L3.14}  that $|{\cal K}_{\epsilon}| < \infty$. Hence, we have that
		\begin{flalign*}
		f_0 - f_{min} &\geq \sum_{k \in {\cal K}_{\epsilon}} (f_{k-1}-f_k) \geq |{\cal K}_{\epsilon}|\rho(H_{Lip} + \sigma_{max})^{-3/2}\epsilon^{3/2}.
		\end{flalign*}
		Rearranging this inequality to yield an upper bound for $|{\cal K}_{\epsilon}|$ we obtain the desired result. 
	\end{proof}
	
	It remains to compute a worst-case upper bound for the number of iterations in ${\cal A}_{\Delta}$ for which ${\cal X}_k$ is larger than a prescribed $\epsilon >0$, and the number of contraction and expansion iterations that may occur between two acceptance steps. We compute these bounds separately in Lemmas \ref{L3.21} and \ref{L3.24}.

	\begin{lemma}\label{L3.21}
		The cardinality of the set ${\cal A}_{\Delta}$ is upper-bounded by
		\begin{equation}\label{3.19}
		\Bigg\lceil \dfrac{f_0 - f_{min}}{\rho \Delta_0^3} \Bigg\rceil \triangleq K_{\Delta} \geq 0.
		\end{equation}
	\end{lemma}
	
	\begin{proof}
		For all $k \in {\cal A}_{\Delta}$, it follows by Lemma \ref{L3.5-3.6} that
		\[f_k - f_{k+1} \geq \rho \|\bs_k\|_2^3 = \rho\Delta_k^3 \geq \rho \Delta_0^3. \]
		Hence, we have that
		\[f_0 - f_{min} \geq \sum_{k=0}^{\infty} (f_k - f_{k+1}) \geq \sum_{k \in {\cal A}_{\Delta}} (f_k - f_{k+1}) \geq |{\cal A}_{\Delta}| \rho \Delta_0^3,\]
		from which the desired result follows.
	\end{proof}
	
	So far, we have obtained upper-bound on the number of acceptance iterations. To obtain  upper-bounds on the number of contraction and expansion iterations that may occur until the next accepted step, let us define, for a given $\hat{k} \in {\cal A}$, 
	\[k_{{\cal A}}(\hat{k}) \triangleq \min \{k \in {\cal A}: k > \hat{k}\}, \]
	\[{\cal I}(\hat{k}) \triangleq \{k \in \mathbb{N}_{+}: \hat{k} <k<k_{{\cal A}}(\hat{k})\}, \]   
	
	Using this notation, the following result shows that the number of expansion iterations between the first iteration and the first accepted step, or between consecutive accepted steps, is never greater than one. Moreover, when such an expansion iteration occurs, it must take place immediately.\\
	
	\begin{lemma}\label{L3.22}
		For any $\hat{k} \in \mathbb{N}_{+}$, if $\hat{k} \in {\cal A},$ then ${\cal E} \cap {\cal I}(\hat{k}) \subseteq \{\hat{k}+1\}$.
	\end{lemma}
	
	\begin{proof}
		By the definition of $k_{\cal A}(\hat{k})$, we have under the conditions of the lemma that ${\cal I}(\hat{k}) \cap {\cal A} = \emptyset $, which means that ${\cal I}(\hat{k}) \subseteq{\cal C} \, \cup \, {\cal E}$. It then follows from Lemma \ref{L3.7} that
		that ${\cal E} \, \cap \, {\cal I}(\hat{k}) \subseteq \{\hat{k}+1\}$, as desired.
	\end{proof}

	\begin{lemma}\label{L3.23}
		For any $k \in \mathbb{N}_{+}$, if $k \in {\cal C}$ and Step~\ref{return-3} in Algorithm~\ref{alg-contract} is reached, then
		\[\dfrac{\lambda_{k+1}}{\|\bs_{k+1}\|_2} \geq \gamma_{\lambda}\Big( \dfrac{\lambda_k}{\|\bs_k\|_2} \Big). \]
	\end{lemma}
	
	\begin{proof}
		If Step~\ref{return-3} is reached, then $\|\bs\|_2 \geq  \gamma_C\|\bs_k\|_2$. It follows that Step~\ref{Sigma-rho_k} in Algorithm~\ref{alg-LC} will produce the primal-dual pair $(\bs_{k+1}, \lambda_{k+1})$ solving $Q_{k+1}$ such that $\|\bs_{k+1}\|_2 = \delta_{k+1}< \|\bs_k\|_2 =\delta_k $ and $\lambda_{k+1} \geq \gamma_{\lambda}\lambda_k$, i.e.,
		\begin{equation}\label{3.20}
		\dfrac{\lambda_{k+1}}{\|\bs_{k+1}\|_2} \geq \dfrac{\gamma_{\lambda}\lambda_k}{\|\bs_k\|_2}.
		\end{equation}
	\end{proof}

	\begin{lemma}\label{L3.24}
		Assume Assumptions \ref{Assumption1-LC} and \ref{Assumption 3} hold. For any $\hat{k} \in \mathbb{N}_{+}$, if $\hat{k} \in {\cal A}$, then
		\[ |{\cal C} \, \cap \, {\cal I}(\hat{k})| \leq  K_{{\cal C}},\]
		where
		\[
		K_\mathcal{C} \triangleq 1+  \left\lceil
		\left(2 + \dfrac{1}{\log(\gamma_{\lambda})}\log \left( \dfrac{\sigma_{max}}{\underline{\sigma}}\right)\right)\dfrac{\log\left(\epsilon^{-1}\Delta\left(C_{min} + \max\{\lambda_{max}, \lambda_0\right)\right)}{\log(1/\gamma_C)} \right\rceil.
		\]
	\end{lemma}
	
	\begin{proof}
		The result holds trivially if $|{\cal C} \cap {\cal I}(\hat{k})| = 0$. Thus, we may assume $|{\cal C} \cap {\cal I}(\hat{k})| \geq 1$. To proceed with our proof, we first claim that the number of iterations $k \in {\cal C}\cap {\cal I}(\hat{k})$ with Step~\ref{return-4} in CONTRACT sub-routine reached, we denote by $K_{{\cal C},1}$, satisfies
		\[K_{{\cal C},1} \leq \dfrac{\log\big(\epsilon^{-1}\Delta(C_{min} + \max\{\lambda_{max}, \lambda_0\big)\big)}{\log(1/\gamma_C)}.\]
		
		By Lemma \ref{LXksk}, and the assumption that ${\cal X}_{k_{\cal A}(\hat{k})} \geq \epsilon$ (optimality not reached yet),
		\[(C_{min} + \max\{\lambda_{max}, \lambda_0\})^{-1}\epsilon \leq \|\bs_{k_{\cal A}(\hat{k})-1}\|_2 \leq \delta_{k_{\cal A}(\hat{k})-1} \leq \gamma_C^{K_{{\cal C},1}}\Delta,\]
		where the last inequality holds by Lemma~\ref{L3.5-3.6}, Lemma~\ref{L3.4}, and the fact that each time Step~\ref{return-4} is reached the radius of the trust region is multiplied by $\gamma_C$. The proof of the claim follows by rearranging the inequality to yield an upper bound for $K_{{\cal C},1}$.\\
		
		It remains to compute the number of iterations between Steps in $k \in {\cal C}\cap {\cal I}(\hat{k})$ for which Step~\ref{return-4} is reached. For a given $\hat{k} \in {\cal A}$, We define 
		\[{\cal I}_{\cal C}(\hat{k})\triangleq \{k \in {\cal C}\cap {\cal I}(\hat{k}): \,\, \mbox{Step~\ref{return-4}}  \mbox{ is reached in Algorithm~\ref{alg-contract}}\} \]
		which correspond to indices in ${\cal C}\cap {\cal I}(\hat{k})$ for which step~\ref{return-4} is reached. Let $k_{{\cal C},1}(\hat{k})$ and $k_{{\cal C},2}(\hat{k})$ be any two consecutive indices in ${\cal I}_{\cal C}(\hat{k})$ with $k_{{\cal C},2}(\hat{k}) - k_{{\cal C},1}(\hat{k}) > 2$. By Lemma~\ref{L3.17}, we have
		
		\[\dfrac{\lambda_{k}}{\|\bs_{k}\|_2} \geq \underline{\sigma}, \quad \forall \; k_{{\cal C},1}(\hat{k}) + 2 \leq k \leq k_{{\cal C},2}(\hat{k}),\]
		which implies that step~\ref{return-3} of CONTRACT subroutine is reached for every $k_{{\cal C},1}(\hat{k}) + 2 < k < k_{{\cal C},2}(\hat{k})$. By Lemmas \ref{L3.18} and \ref{L3.23}, we then get
		\[\sigma_{max}  \geq \dfrac{\lambda_{k_{{\cal C},2}(\hat{k})}}{\|\bs_{k_{{\cal C},2}(\hat{k})}\|_2} \geq \underline{\sigma}\gamma_{\lambda}^{k_{{\cal C},2}(\hat{k}) - k_{{\cal C},1}(\hat{k}) - 2}.\]
		
		Hence,
		\[k_{{\cal C},2}(\hat{k}) - k_{{\cal C},1}(\hat{k}) \leq \dfrac{1}{\log(\gamma_{\lambda})}\log \big( \dfrac{\sigma_{max}}{\underline{\sigma}}\big)+2.\]
		
		Now let $k_{{\cal C},last}(\hat{k}) - k_{{\cal C},1}(\hat{k})$ be the last element of ${\cal I}_{\cal C}(\hat{k})$. Similarly, we can show that 
		\[k_{{\cal A}}(\hat{k}) - k_{{\cal C},last}(\hat{k}) \leq \dfrac{1}{\log(\gamma_{\lambda})}\log \big( \dfrac{\sigma_{max}}{\underline{\sigma}}\big)+2.\]
		The desired result follows s since $|{\cal C} \, \cap \, {\cal I}(\hat{k})| = 1 + K_{C,1}\left( \dfrac{1}{\log(\gamma_{\lambda})}\log \big( \dfrac{\sigma_{max}}{\underline{\sigma}}\big)+2\right) $.
	\end{proof}

	Notice that since $\underline{\sigma} = {\cal O}(\epsilon^{-1})$, the number of contract steps $K_C$ is of order ${\cal O}(\log^2 \epsilon^{-1})$. Due to the while loop in Step~\ref{while-loop} of the CONTRACT sub-routine, completing a contract step may require solving more than one sub-problem. Our next result provides an upper bound on the number of subproblem routine calls required in one contract step.
	
	\begin{lemma}\label{K_C^1}
		Assume Assumptions \ref{Assumption1-LC} and \ref{Assumption 3} hold. For a scalar $\epsilon \in (0,\infty)$, the total number of sub-problems we are required to solve in a step $k \in {\cal C}$ with ${\cal X}_k \geq \epsilon$, is at most
		\[K_{\cal C}^{1} \triangleq \log \Big((C_{min} + \max\{\lambda_{max}, \lambda_0\}\big)\dfrac{H_{max} + g_{Lip} + \rho \Delta}{\underline{\sigma}\epsilon}\Big).\]
	\end{lemma}
	
	\begin{proof}
		We prove the result by contradiction. Assume the contrary,
		\begin{flalign*}
		\lambda &\geq \lambda_k\gamma_{\lambda}^{\log \Big(\big(C_{min} + \max\{\lambda_{max}, \lambda_0\}\big) \dfrac{H_{max} + g_{Lip} + \rho \Delta}{\underline{\sigma}\epsilon}\Big)}\\
		&\geq \underline{\sigma}\|\bs_k\|_2 \big(C_{min} + \max\{\lambda_{max}, \lambda_0\}\big) \dfrac{H_{max} + g_{Lip} + \rho \Delta}{\underline{\sigma}\epsilon}\\
		&\geq H_{max} + g_{Lip} + \rho \Delta,
		\end{flalign*}
		where the last inequality holds by Lemma \ref{LXksk} and the fact that ${\cal X}_k \geq \epsilon$. Hence, by Lemma \ref{L3.8}
		\[\dfrac{f(\bx_k + \bs) - f_k}{\|\bs\|_2^3} \geq \rho > \dfrac{f(\bx_k + \bs_k) - f_k}{\|\bs_k\|_2^3},\]
		where $\bs$ is computed by solving $Q_k(\lambda)$ and the strict inequality holds since $k \in {\cal C}$. We conclude that $\|\bs_k\|_2 \neq \|\bs\|_2$ which contradicts the condition of the while loop. This completes the proof.
	\end{proof}
	Notice that from the definition of $\underline{\sigma}$ in the algorithm, $K_{\cal C}^1 = {\cal O}(\log \epsilon^{-1})$.

	\begin{theorem}\label{L3.25}
		Under Assumptions \ref{Assumption1-LC} and \ref{Assumption 3}, for a scalar $\epsilon \in (0,\infty)$, the total number of elements in the index set
		\[ \{k \in \mathbb{N}_{+} : {\cal X}_k >\epsilon\}\]
		is at most
		\begin{equation}\label{3.22}
		K(\epsilon) \triangleq 1 + (K_{\sigma}(\epsilon) + K_{\Delta})(1 + K_{{\cal C}}K_{\cal C}^1),
		\end{equation}
		where $K_{\sigma}(\epsilon)$, $K_{\Delta}$, $K_{\cal C}^{1}$, and $K_{{\cal C}}$ are defined in Lemmas \ref{L3.20}, \ref{L3.21}, \ref{L3.24}, and  \ref{K_C^1} respectively. Hence, for $\epsilon_g>0$, the number of sub-problem routines required for First-Order-LC-TRACE to find an $\epsilon_g$-first-order stationary point is at most $K(\epsilon_g) = {\cal O}\Big(\epsilon_g^{-3/2}\log^3 \epsilon_g^{-1}\Big)$.
	\end{theorem}
	
	\begin{proof}
		Without loss of generality, we may assume that at least one iteration is performed. Lemmas \ref{L3.20} and \ref{L3.21} guarantee that the total number of elements in the index set $\{k \in {\cal A} \, | \, k \geq 1, \, {\cal X}_k > \epsilon\}$ is at most $K_{\sigma}(\epsilon)+K_{\Delta}$. Also, immediately prior to each of the corresponding accepted steps, Lemmas \ref{L3.22}, \ref{L3.24}, and \ref{K_C^1} guarantee that at most $1 + K_{C}K_{C}^1$ sub-problem routine calls are required in expansion and contraction. Accounting for the first iteration, the desired result follows.
	\end{proof}

	\subsection{Proof of Theorem~\ref{epsilon-second-order-convergence-complexity}}\label{epsilon-second-order-convergence-complexity-app}.
	
	\begin{proof}
		Let us first define
		\[ {\cal K}_{\epsilon}\triangleq \{k \in \mathbb{N}_{+} : k\geq 1; \, (k-1) \in {\cal A}_{\sigma}; \, {\cal X}_k >\epsilon\}\]
		and
		\[{\cal V} \triangleq \{k \, | \, \mbox{Step~\ref{alg-first-order-reached} in Algorithm~\ref{alg-GLC-cap} is reached at iteration } k\}.\]
		
		To proceed with our proof, we first show that if $k \in {\cal V} \, \cup \,  {\cal K}_{\epsilon}$, the following reduction bound on the objective value holds
		\begin{equation}
		\label{eq:Reduction}
		f_{k-1} - f_{k+1} \geq \min\left\{\dfrac{\rho \cX_k^{3/2} }{(H_{Lip} + \sigma_{max})^{3/2}}, \dfrac{2\cpsi_k^3}{3\tilde{H}^2}\right\}
		\end{equation} 
		
		If $k \in {\cal V}$, then using second-order descent lemma, we obtain
		\begin{flalign}
		f_{k+1} &\leq f_k + \langle \bg_k, \bx_{k+1}- \bx_k \rangle + \dfrac{1}{2}(\bx_{k+1}- \bx)^T \bH_k (\bx_{k+1}- \bx_k) +  \dfrac{H_{Lip}}{6}\|\bx_{k+1}- \bx_k\|_2^3 \nonumber\\
		&\leq f_k + \langle \bg_k, \bx_{k+1}- \bx_k \rangle + \dfrac{1}{2}(\bx_{k+1}- \bx_k)^T\bH_k (\bx_{k+1}- \bx_k) +  \dfrac{\tilde{H}}{6}\|\bx_{k+1}- \bx_k\|_2^3 \nonumber \\
		&= f_k + \dfrac{2\cpsi_k}{\tilde{H}}\langle \bg_k, \widehat{\bd}_k \rangle + \dfrac{2\cpsi_k^2}{\tilde{H}^2}(\widehat{\bd}_k)^T \bH_k (\widehat{\bd}_k) + \dfrac{4\cpsi_k^3}{3\tilde{H}^2}\|\widehat{\bd}_k\|_2^3 \nonumber\\
		&\leq f_k - \dfrac{2\cpsi_k^3}{\tilde{H}^2} + \dfrac{4\cpsi_k^3}{3\tilde{H}^2} \nonumber\\
		&= f_k - \dfrac{2\cpsi_k^3}{3\tilde{H}^2}.\label{reduction-d-LC}
		\end{flalign}
		Now if a first-order stationary point was reached at $k-1$, i.e. $k-1 \in {\cal V}$, it directly follows from~\eqref{reduction-d-LC} that
		\begin{equation}\label{reduction-second-order-1}
		f_{k-1} - f_{k+1} \geq f_k - f_{k+1} \geq \dfrac{2\cpsi_k^3}{3\tilde{H}^2}, \quad \forall \,\, k \in {\cal V}, \, k-1 \in {\cal V}.
		\end{equation}
		Otherwise, Step~\ref{alg-first-order-not-reached} of Algorithm~\ref{alg-GLC-cap} was reached at iteration $k-1$ and First-Order-LC-TRACE was called. The former algorithm by definition is a monotone algorithm, i.e. it generates a sequence of iterates for which the corresponding sequence of objective values is decreasing. This property combined with~\eqref{reduction-d-LC} implies that
		\begin{equation}\label{reduction-second-order-2}
		f_{k-1} - f_{k+1} \geq f_k - f_{k+1} \geq \dfrac{2\cpsi_k^3}{3\tilde{H}^2}, \quad \forall \,\, k \in {\cal V}, \, \mbox{First-Order-LC-TRACE called at } k-1.
		\end{equation}  
		Combining~\eqref{reduction-second-order-1} and \eqref{reduction-second-order-2}, we get
		\begin{equation}\label{reduction-second-order}
		f_{k-1} - f_{k+1} \geq f_k - f_{k+1} \geq \dfrac{2\cpsi_k^3}{3\tilde{H}^2}, \quad \forall \,\, k \in {\cal V}.
		\end{equation}  
		
		We next show a lower bound on the reduction of the objective value if $k \in {\cal K}_{\epsilon}$. By Lemma~\ref{L3.20}, we have
		\begin{equation}\label{reduction-A-sigma}
		f_{k-1} - f_{k+1} \geq f_{k-1} - f_{k} \geq \dfrac{\rho \cX_k^{3/2} }{(H_{Lip} + \sigma_{max})^{3/2}},
		\end{equation}
		where the first inequality again holds by the monotonicity of First-Order-LC-TRACE and~\eqref{reduction-d-LC}. Combining~\eqref{reduction-second-order} and \eqref{reduction-A-sigma}, we get
		\[
		f_{k-1} - f_{k+1} \geq \min\left\{\dfrac{\rho \cX_k^{3/2} }{(H_{Lip} + \sigma_{max})^{3/2}}, \dfrac{2\cpsi_k^3}{3\tilde{H}^2}\right\}, \quad \forall k \in \mathcal{K}_{\epsilon} \cup \mathcal{V}. 
		\]
		By summing over the iterations we get 
		\begin{flalign*}
		2(f_0 - f_{min}) &\geq \sum_{k \in {\cal K}_{\epsilon} \cup {\cal V}, k \geq 1} (f_{k-1}-f_{k+1}) \geq |{\cal K}_{\epsilon} \cup {\cal V}|\min\left\{\dfrac{\rho \cX_k^{3/2} }{(H_{Lip} + \sigma_{max})^{3/2}}, \dfrac{2\cpsi_k^3}{3\tilde{H}^2}\right\}.
		\end{flalign*}
		Rearranging this inequality yields
		\begin{equation}\label{reduction-1}
		|{\cal K}_{\epsilon} \cup {\cal V}| \leq \dfrac{2(f_0 - f_{min})}{\max\left\{\dfrac{\rho \cX_k^{3/2} }{(H_{Lip} + \sigma_{max})^{3/2}}, \dfrac{2\cpsi_k^3}{3\tilde{H}^2}\right\}}
		\end{equation}
		
		Let ${\cal H}(\epsilon_g, \epsilon_H) \triangleq \{k  \, \, | \, \, \cX_k >\epsilon_g,\,\, \cpsi_k >\epsilon_H\}.$ Using Lemmas \ref{L3.22}, \ref{L3.24} and \ref{K_C^1}, the number of sub-problem routine calls required in expansion and contraction between two acceptance steps is upper bounded by $1 + K_{{\cal C}}K_{\cal C}^1$. Hence,  
		\[|{\cal H}(\epsilon_g, \epsilon_H)| \leq (|{\cal K}_{\epsilon} \cup {\cal V}| + K_{\Delta})(1 + K_{{\cal C}}K_{\cal C}^1),\]
		which concludes our proof.
		
	\end{proof}

\end{document}